\documentclass{amsart}

\usepackage{graphics}
\usepackage{thmtools}
\usepackage{wasysym}
\usepackage[T1]{fontenc}    

\usepackage[export]{adjustbox}
\usepackage{paralist}

\usepackage{amsthm}
\usepackage{amsbsy,amsmath,amssymb,amscd,amsfonts}
\usepackage[pagebackref=false]{hyperref}
\usepackage[nameinlink,capitalize,noabbrev]{cleveref}

\usepackage{graphicx,float,latexsym,color}
\usepackage[font={scriptsize,it}]{caption}
\usepackage{subcaption}

\usepackage{makecell}

\usepackage[dvipsnames]{xcolor}

\newtheorem{theorem}{Theorem}
\newtheorem*{definition*}{Definition}
\newtheorem*{theorem*}{Theorem}
\newtheorem{observation}{Observation}

\newtheorem{challenge}{Challenge}
\newtheorem{proposition}{Proposition}
\newtheorem{conjecture}{Conjecture}
\newtheorem{corollary}{Corollary}

\theoremstyle{remark}
\newtheorem{remark}{Remark}
\theoremstyle{definition}

\hypersetup{
    pdftoolbar=true,        
    pdfmenubar=true,        
    pdffitwindow=false,     
    pdfstartview={FitH},    
    colorlinks=true,       
    linkcolor=OliveGreen,          
    citecolor=blue,        
    filecolor=black,      
    urlcolor=red           
}

\usepackage{lineno}

\usepackage{comment}
\usepackage{microtype}
\usepackage{footnote}

\newcommand{\E}{\mathcal{E}}
\newcommand{\T}{\mathcal{T}}
\newcommand{\C}{\mathcal{C}}

\renewcommand{\P}{\mathcal{P}}
\newcommand{\D}{\mathcal{D}}
\renewcommand{\T}{\mathcal{T}}

\renewcommand{\L}{\mathcal{L}}
\renewcommand{\S}{\mathcal{S}}

\title{Poncelet Parabola Pirouettes}
\author[D. Reznik]{Dan Reznik}
\thanks{D. Reznik$^*$, Data Science Consulting, Rio de Janeiro, Brazil. \texttt{dreznik@gmail.com}}
\author[R. Garcia]{Ronaldo Garcia}
\thanks{R. Garcia$^*$, Fed. Univ. of Goiás, Goiânia, Goiás, Brazil. \texttt{ragarcia@ufg.br}}

\begin{document}

\maketitle

\begin{abstract}

We describe some three-dozen curious phenomena manifested by parabolas inscribed or circumscribed about certain Poncelet triangle families. Despite their pirouetting motion, parabolas' focus, vertex, directrix, etc., will often sweep or envelop rather elementary loci such as lines, circles, or points. Most phenomena are unproven though supported by solid numerical evidence (proofs are welcome). Some yet unrealized experiments are posed as ``challenges'' (results are welcome!).

\end{abstract}

\section{Introduction}

We visit three-dozen surprising Euclidean phenomena manifested by parabolas dynamically inscribed or circumscribed about Poncelet families of triangles. As shown in  \cref{fig:poncelet-tris}, these are triangles simultaneously inscribed and circumscribed about two conics \cite{dragovic11}. Examples of works exploring loci and invariants of Poncelet triangle families include \cite{odehnal2011-poristic,pamfilos2020,reznik2020-intelligencer,skutin2013-isogonal,zaslavsky2003-trajectories}. 
The references   used in this paper with respect to classic concepts and
facts, are not linked to the original sources; only specific contributions are listed
as articles and directly linked to their sources.
\begin{figure}
    \centering
    \includegraphics[width=\textwidth]{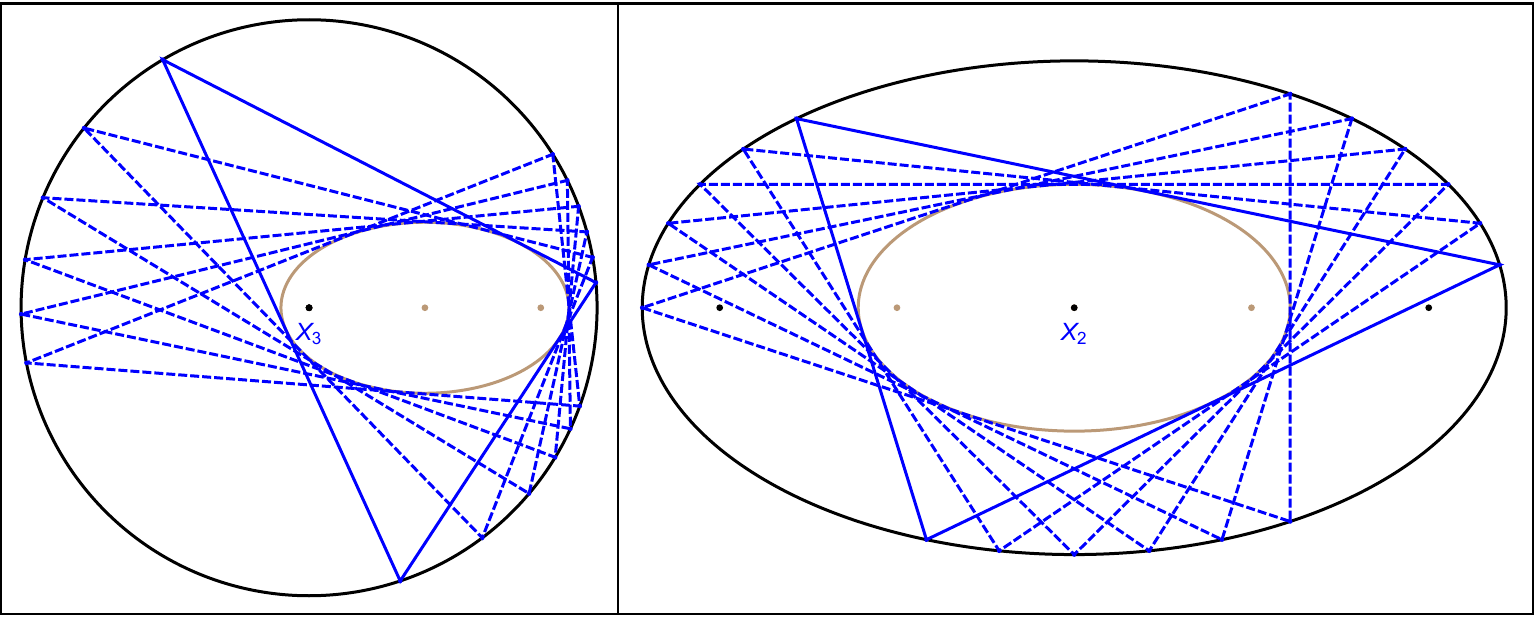}
    \caption{\textbf{Left:} Poncelet triangles inscribed in a circle (fixed circumcircle) and circumscribing a conic, i.e., the ``caustic''. In the case shown, one of the caustic's focus is the circumcenter $X_3$. \textbf{Right:} Poncelet triangles interscribed between two concentric, homothetic ellipses, where the outer (resp. inner) is the Steiner circumellipse (resp. inellipse), both of which are centered on the barycenter $X_2$.}
    \label{fig:poncelet-tris}
\end{figure}

Referring to \cref{fig:circumparabolas-basic}, every triangle is associated with a 1d family of {\em circumparabolas} which pass through the three vertices. These can be swept as (i) the image under isogonal conjugation of lines tangent to the circumcircle, or (ii) as the image under isotomic conjugation of lines tangent to the Steiner ellipse\footnote{This is  the unique circumellipse centered on the barycenter $X_2$ \cite[Circumconic]{mw}.}. For details on both isotomic and isogonal conjugation, see \cite{akopyan2007-conics}, \cite{garcia2021-impa} and \cref{app:conjug}.

\begin{figure}
    \centering
    \includegraphics[width=\textwidth]{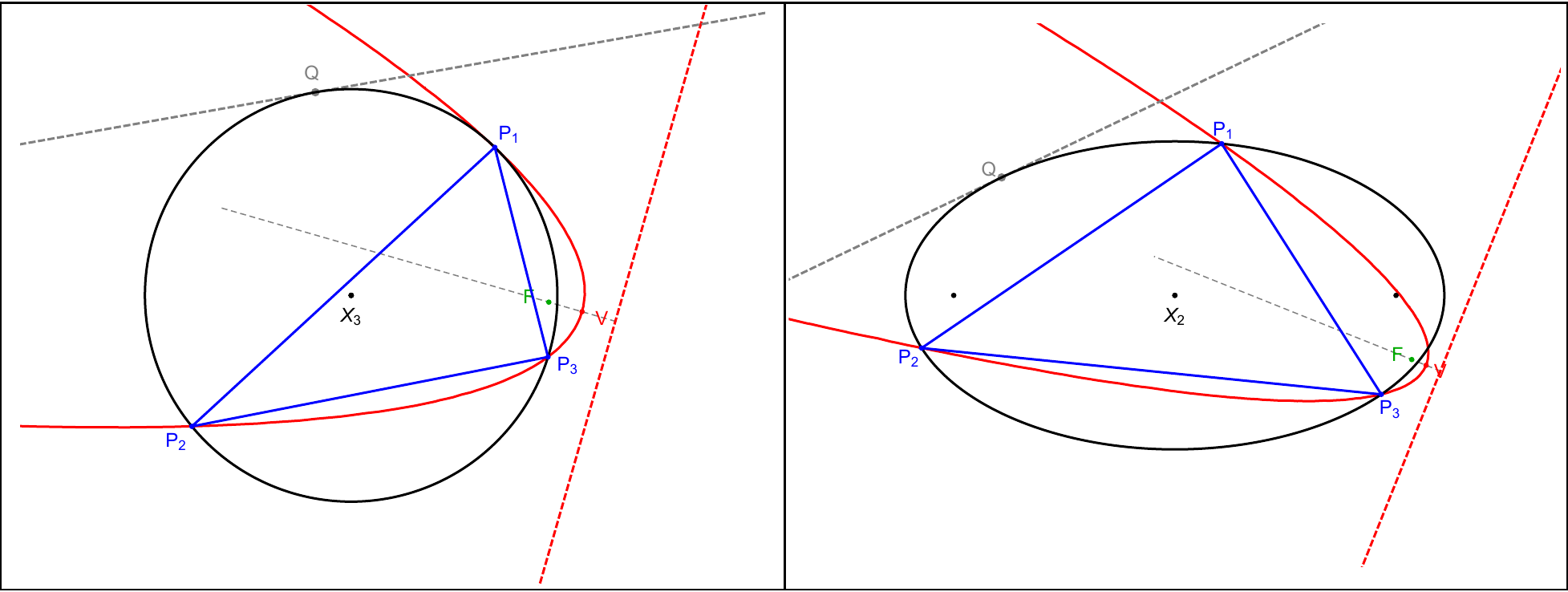}
    \caption{\textbf{Left:} A triangle's circumparabola (red) passes thru the 3 vertices and is the isogonal image of a line tangent to the circumcircle at a point $Q$ \cite[circumconic]{mw}. Also shown is the vertex $V$ and the directrix (dashed red). \textbf{Right:} Alternatively, a circumparabola is also the isotomic image of a line tangent to the Steiner ellipse at a point $Q$ \cite[isotomic conjugate]{mw}.}
    \label{fig:circumparabolas-basic}
\end{figure}

Similarly, every triangle is associated with a 1d family of inscribed parabolas or  {\em inparabolas}, tangent to each of the sidelines, see \cref{fig:inparabolas-basic-left} and \cref{fig:inparabolas-basic-right}.
   
The focus $F$ (resp. Brianchon\footnote{This is the perspector of a triangle and an inconic \cite{mw}.} point $\Pi$) always lies on the circumcircle (resp. Steiner ellipse) \cite{mw}. So to generate all inparabolas one can either (i) sweep $F$ over the circumcircle, or (ii) sweep $\Pi$ over the Steiner circumellipse.

\begin{figure}
    \centering
    \includegraphics[width=0.6\textwidth]{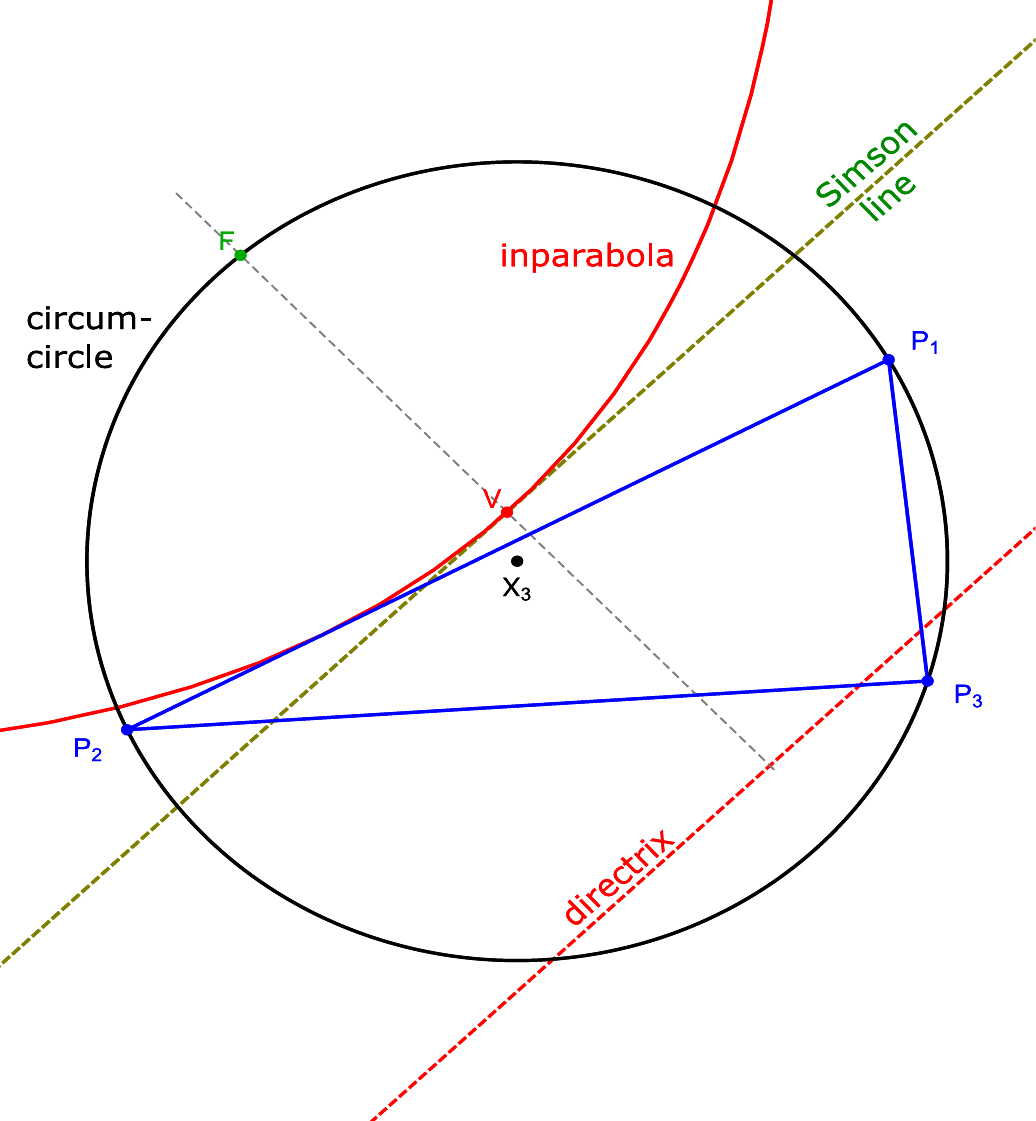}
    \caption{An inparabola (red) is tangent to a reference triangle's sides (blue). Its focus $F$ lies on the circumcircle \cite[inconic]{mw}. Also shown is the vertex $V$ and the directrix (dashed red). The latter is parallel to the $F$-Simson line $\S$ which passes through $V$ \cite{akopyan2007-conics}. }
    \label{fig:inparabolas-basic-left}
\end{figure}

\begin{figure}
    \centering
    \includegraphics[width=0.8\textwidth]{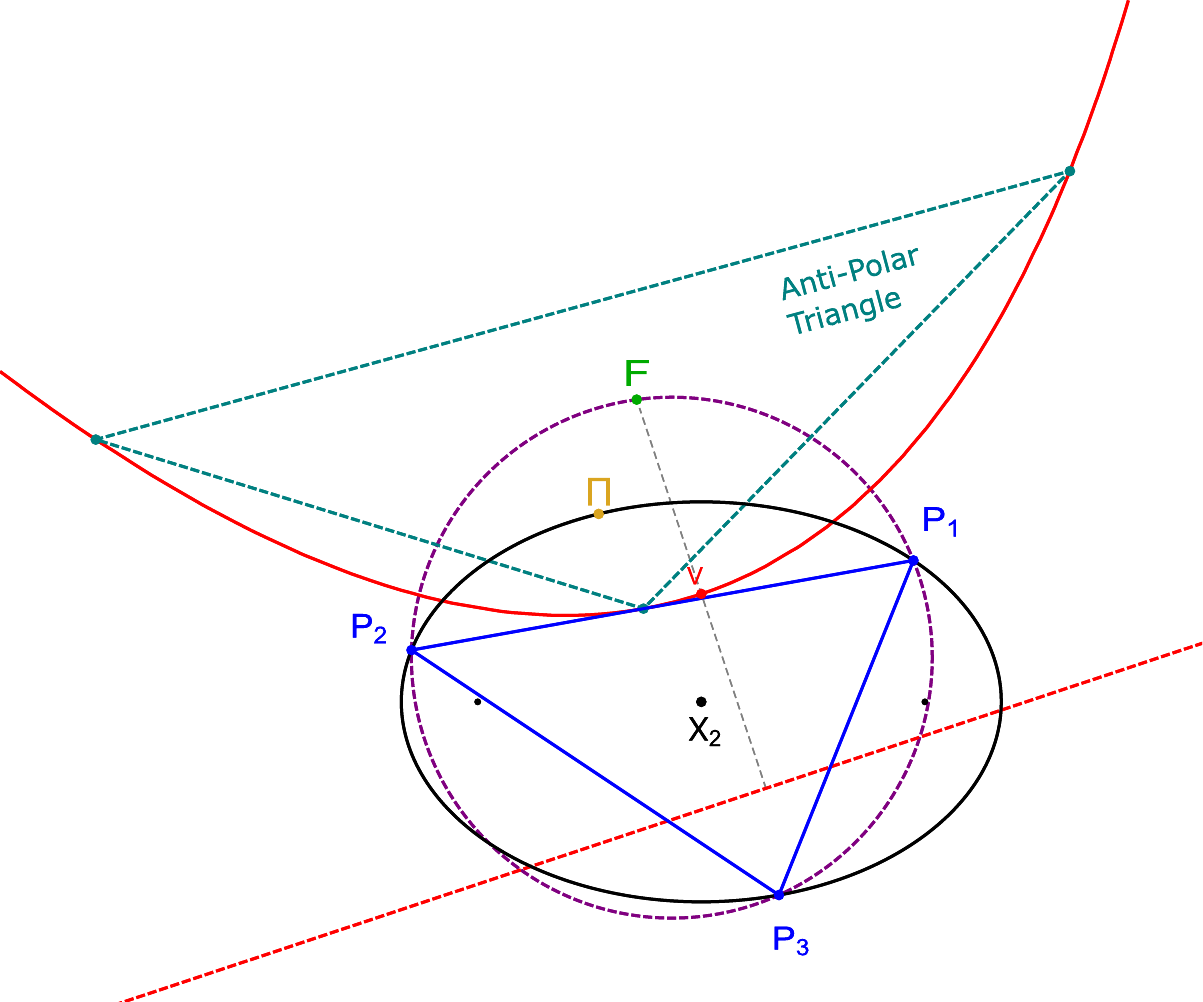}
    \caption{  The anti-polar triangle $T'$ (dashed teal) of a reference triangle $T$ (blue with respect to an inparabola $\P$ (red) has vertices at the touchpoints of $\P$ on the sidelines of $T$ (i.e., $T$ is the polar of $T'$ with respect to $\P$). Since $\P$ is an inparabola of $T$, its focus $F$ lies on the circumcircle. 
    In \cite[TC7(2)]{stothers-perspectors} it was proved   that $T'$ and $T$ are in perspective at a point $\Pi$ which lies on the Steiner ellipse (black).
    }
    \label{fig:inparabolas-basic-right}
\end{figure}

\subsection*{Experimental Thrust and a Preview of Results}

Fueled by much curiosity and using tools of graphical simulation (and numerical verification second), we look for salient phenomena manifested by in- or circumparabolas to certain ``hand-picked'' Poncelet families, namely, where the outer conic is either a circle or the Steiner ellipse itself, see \cref{fig:poncelet-tris,fig:app-poncelet-circle-inscribed}.


Specifically, for circle- (resp. Steiner-) inscribed Poncelet, we fix the focus (resp. Brianchon point) on the outer conic. As we traverse Poncelet triangles in a given family, we observe that parabolas' traditional accessories such as the vertex, perspector, directrix, polar triangle, will often sweep (or envelop) simple curves such as conics, circles, lines, and/or points. This is similar in spirit to \cite{skutin2013-isogonal}.

In turn, this has driven us to document these results, which are in their majority presented below as (unproven) observations.
The results are based on numerical and graphical experiments with help of computer systems.
When certain patterns emerge over several families, we generalize them:  \cref{conj:ip-circum-vtx}, \cref{conj:ip-W},
\cref{conj:ip-locus-W}, \cref{conj:cp-isog-directrix}. 

In \cite{odehnal2022-parabolas}, an algebro-geometric proof is provided for \cref{thm:focus}, and new related results concerning the envelopes and loci of circumparabolas are demonstrated.

\subsection*{Article structure}

In 
\cref{sec:ip-circle-inscribed,sec:ip-steiner-inscribed}
we describe 
inparabola phenomena over both circle- and Steiner-inscribed Poncelet families. 
\cref{sec:cp-isog,sec:cp-isot}
focus on
circumparabola 
phenomena, over similarly-inscribed triangle families. A summary appears in \cref{sec:summary} as well as a link to narrated videos of some experiments.
See YouTube playlist \cite{playlist2021-parabolas}.
In \cref{app:poncelet} the four circle-inscribed  Poncelet families studied are reviewed.
In \cref{app:conjug} the geometry of isogonal and isotomic conjugation is reviewed. In \cref{app:explicit} we derive explicit formulas for a triangle's circum- and inparabola.

 \section{Inparabolas over Circle-Inscribed Poncelet}
\label{sec:ip-circle-inscribed}

In this section we describe loci and envelope phenomena manifested by inparabolas $\P$ of circle-inscribed Poncelet families (\cref{fig:app-poncelet-circle-inscribed}), such that their focus $F$ is a fixed point on the circumcircle. Below, let the ``reflection'' of a point $A$ about $O$ be a point $A'$ such that the latter is the midpoint of $AA'$. 

Let $V$ (resp. $C$) denote the vertex of $\P$ (resp. the reflection of $F$ about $V$, i.e., the projection of $F$ or $V$ on the directrix), see \cref{fig:ip-circum-loci}. It can be shown the Simson line\footnote{The feet of perpendiculars (i.e., the pedal triangle) dropped from any point $F$ on the circumcircle onto the sides of a triangle are collinear on an line known as the ``Simson line'' \cite{mw}.}  $\S$ of a triangle with respect to $F$ is parallel to the directrix and tangent to $\P$ at $V$, $V$ is the projection of $F$ on said line \cite{akopyan2021-private,gheorghe2021-private}. So any properties of $V$ mentioned below are properties of projections of $F$ on $\S$.

Gallatly shows that the envelope of Simson lines over the bicentric family is a point \cite{gallatly1914-geometry}.


\subsection{The inellipse family}

The inellipse family appears in \cref{fig:app-poncelet-circle-inscribed}(top left).  Referring to \cref{fig:ip-circum-loci}, over this family, one observes:

\begin{observation}
The locus of $V$ is a circle passing through $F$ and tangent to the inellipse (Poncelet caustic) at the antipode $U$ of $F$ on the locus.
\label{obs:ip-circum-locus}
\end{observation}

Let $\rho$ denote the radius of the locus of $V$.

\begin{corollary}
The locus of $C$ is a circle of radius $2\rho$ centered on $U$.
\end{corollary}

Still referring to \cref{fig:ip-circum-loci}, let $W$ denote the reflection of $F$ about $U$. Since $V$ lies on a circle with $FU$ as a diameter (a numerical observation), $\triangle F V U$ is a right triangle. Since the Simson line is tangent to the inparabola at $V$ is must pass though $U$. The same is true for the directrix (it must pass through $W$. Therefore:

\begin{corollary}
Over the family, the envelope of the directrix (resp. Simson line) is $W$ (resp. $U$).
\end{corollary}

\begin{figure}
    \centering
    \includegraphics[trim=100 0 50 0 ,clip,width=.8\textwidth]{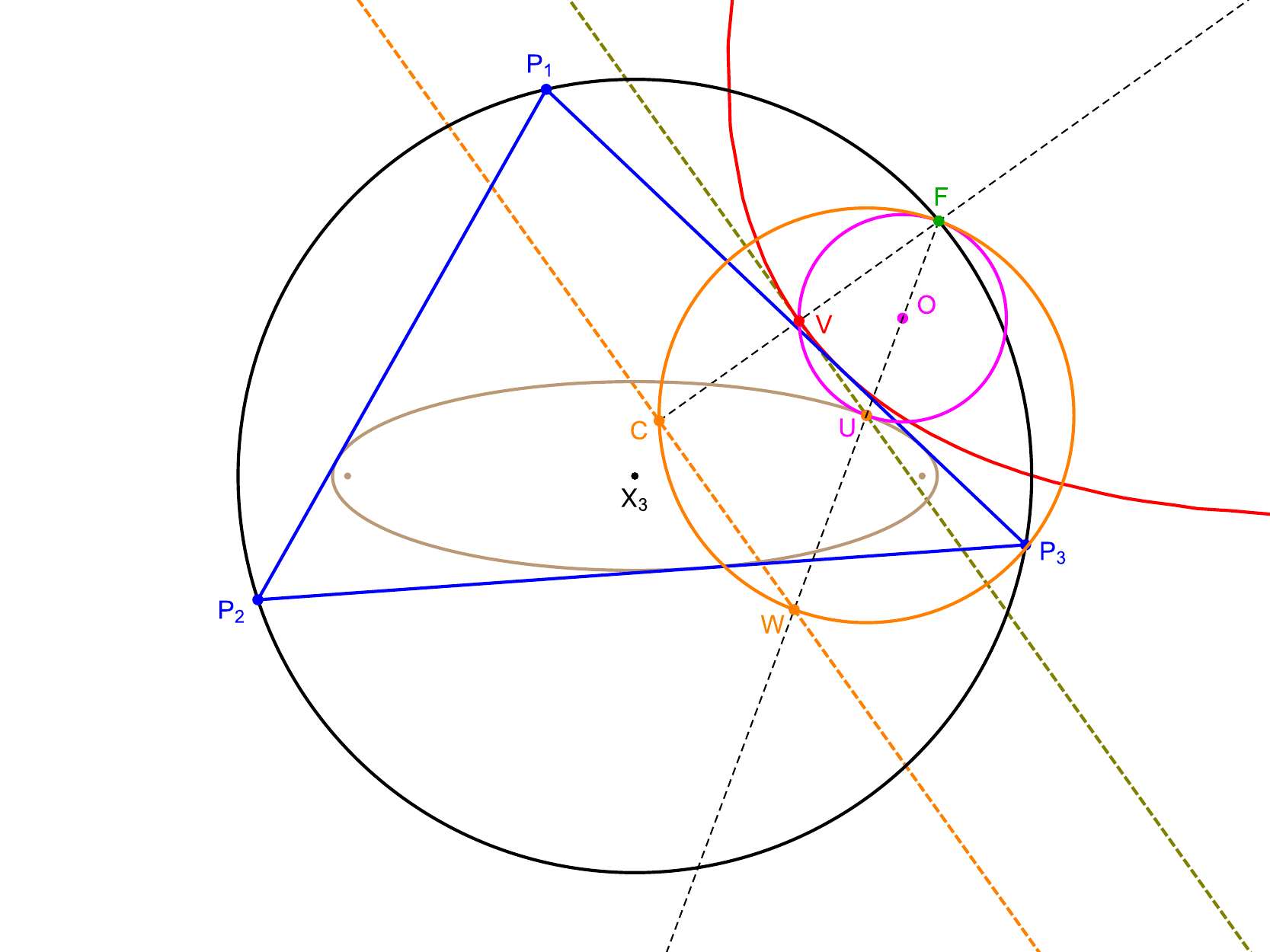}
    \caption{A Poncelet triangle (blue) is shown of the ``inellipse'' family, as well as the inparabola $\P$ (red) with focus at a fixed point $F$ on the circumcircle; let $V$ and $C$ denote the vertex of $\P$ and its projection on the directrix (dashed orange), respectively. The triangle's $F$-Simson line \cite{mw} $\S$ (dark green) is parallel to the directrix and tangent to $\P$ at $V$. Over the Poncelet family, (i) the locus of $V$ is a circle (magenta) passing through $F$ and tangent to the caustic at a point $U$; $O$ indicates its center. (ii) The locus of $C$ is a twice-sized circle (orange) which also contains $F$, and whose center is $U$. Let $W$ be the reflection of $F$ about $U$. Over the family, the directrix (resp. $F$-Simson line) pass through fixed $W$ (resp. fixed $U$).}
    \label{fig:ip-circum-loci}
\end{figure}

\subsubsection*{Over all foci} We can regard \cref{obs:ip-circum-locus} as associating with each $F$ a circular locus, and more specifically, the center $O$ of that locus, as well as a fixed point $W$ about which the directrix turns. Referring to \cref{fig:ip-circum-over-F}:

\begin{corollary}
Over all $F$ on the circumcircle, the locus of the touch-point $U$ of the circular locus of inparabola vertices is the caustic itself.
\end{corollary}

\begin{observation}
Over all $F$ on the circumcircle, the locus of $O$ is an ellipse concentric and axis-aligned with the caustic of the inellipse family.
\end{observation}

\begin{observation}
Over all $F$ on the circumcircle, the locus of $W$ is a circle concentric with the two Poncelet conics.
\end{observation}

\begin{figure}
    \centering
    \includegraphics[width=\textwidth]{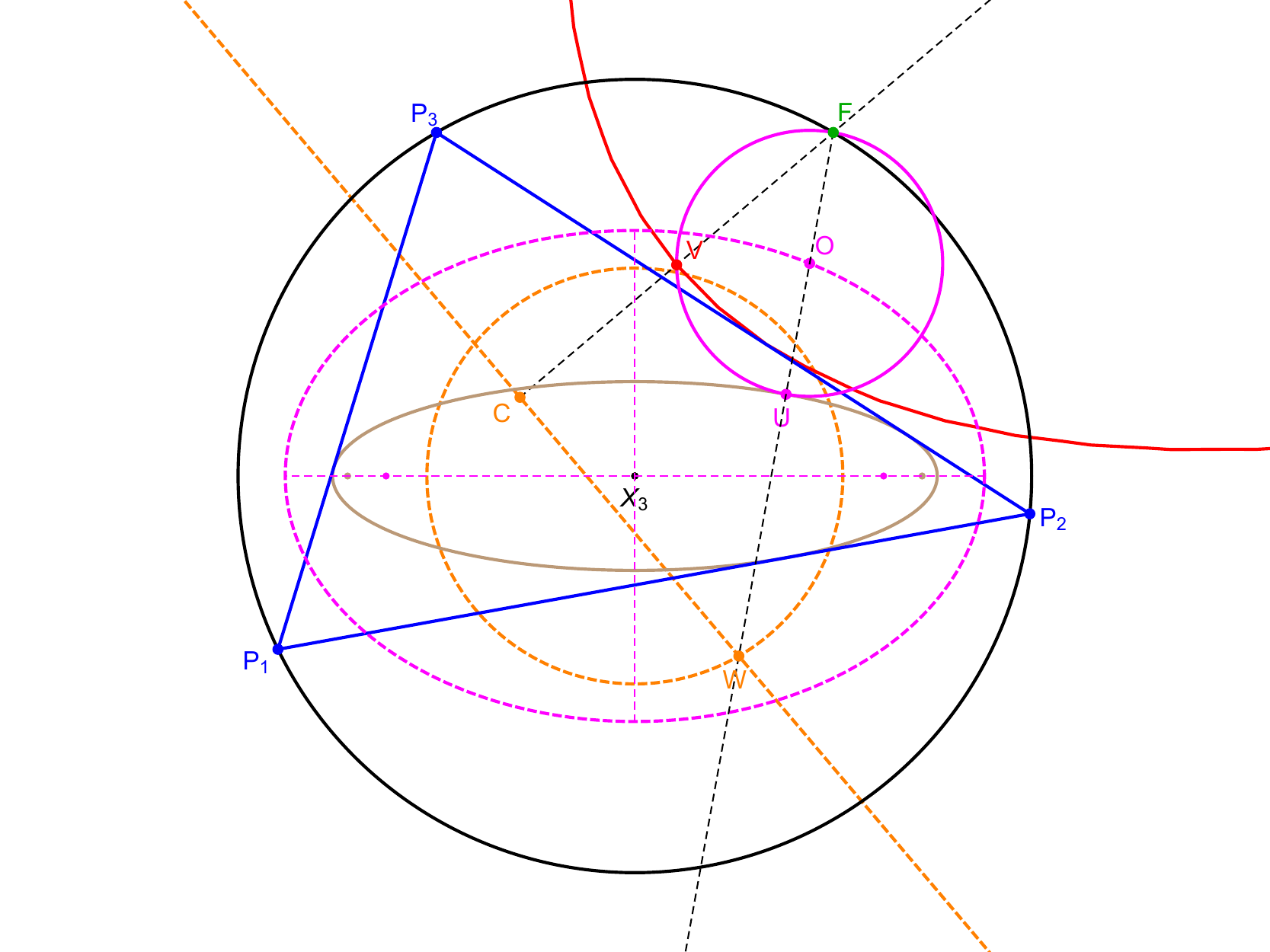}
    \caption{Over all $F$ on the circumcircle, the locus of the center $O$ of the circular locus of the vertex (magenta circle) is an ellipse (dashed magenta), concentric and axis-aligned with the caustic. Over all $F$, the locus of $W$ (envelope of the directrix) is a concentric circle (dashed orange).}
    \label{fig:ip-circum-over-F}
\end{figure}

\subsection{Bicentric family}

Referring to \cref{fig:ip-bic}(left), all observations pertaining to the circumcircle family remain true, namely:

\begin{observation}[Bicentric combo]
Over the Bicentric family, the locus of both $V$ and $C$ are circles, and all Simson lines (resp. directrices) pass through a fixed point $U$ (resp. $W$), where $U$ is antipodal to $F$ on the locus of $V$, and $W$ is the reflection of $F$ about $U$.
\label{obs:bic-combo}
\end{observation}

Notice that unlike the case of the inellipse family, here the locus of $V$ is not tangent to the caustic. Referring to \cref{fig:ip-bic} (right), over all $F$:

\begin{observation}
Over all foci $F$ of inparabolas, the locus of the center $O$ of the (circular) loci of the vertex is an ellipse whose minor axis runs along $X_1 X_3$ and whose center is that segment's midpoint $X_{1385}$.
\end{observation}

\begin{observation}
The locus of $U$ is an ellipse internally tangent to the caustic, with minor axis along $X_1 X_3$, and centered on $X_1$.
\end{observation}

\begin{observation}
The locus of $W$ is a circle with center on the $X_1 X_3$ axis.
\end{observation}

\begin{figure}
    \centering
    \includegraphics[width=\textwidth]{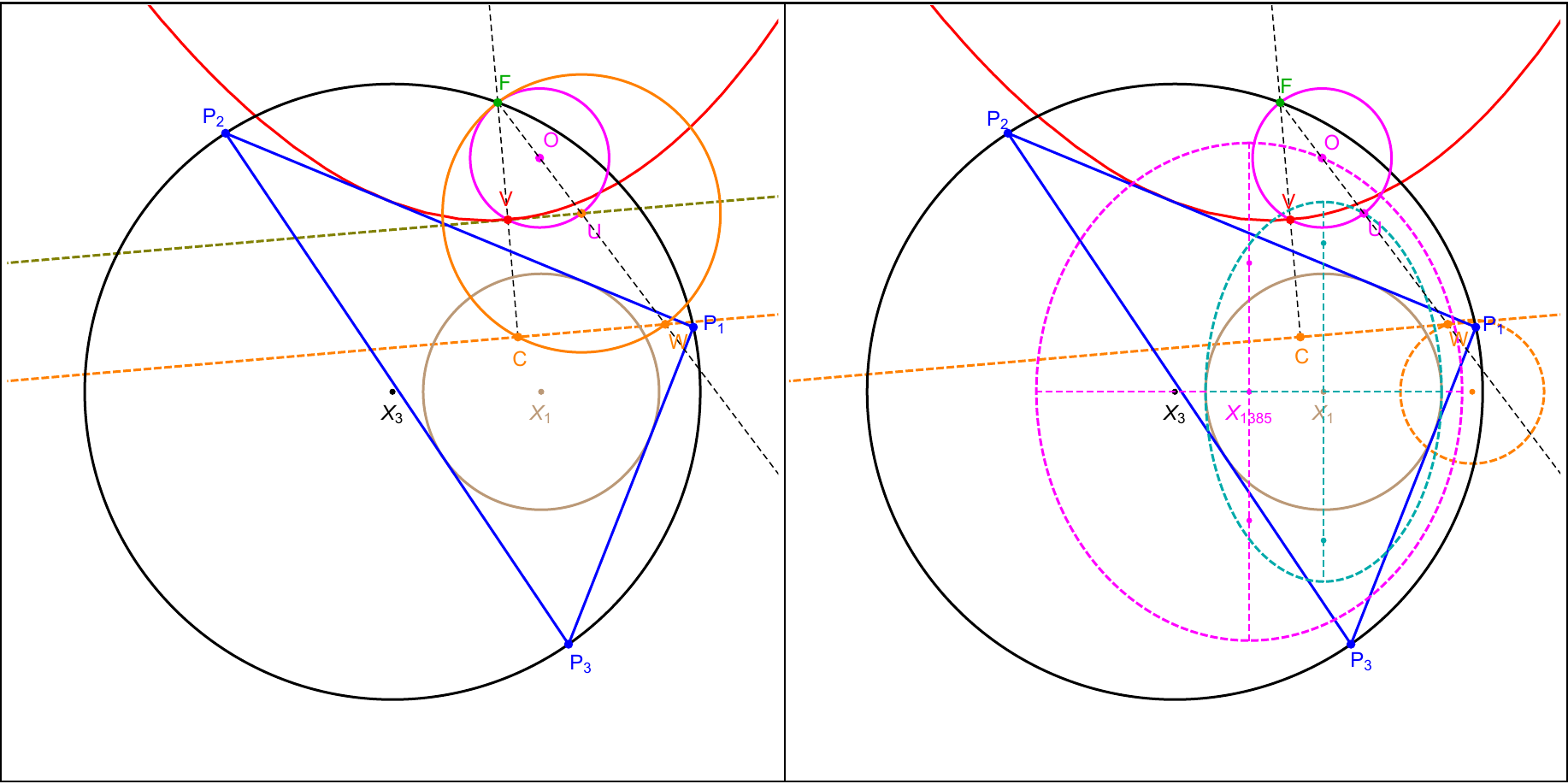}
    \caption{\textbf{Left:} Shown is a Poncelet triangle (blue) in the bicentric family. Consider inparabolas $\P$ (red) with focus a fixed point $F$ on the circumcircle. As in the inellipse family, the locus of both $V$ and $C$ are circles containing $F$ (magenta and orange); as before, over the family, the directrix passes through a fixed point $W$ which is diametrically opposite to $F$ on the $C$ locus. \textbf{Right:} over all $F$ on the circumcircle, the locus of the center $O$ of the circular locus of the vertex, is an ellipse (dashed pink) with minor axis along the $X_1 X_3$ line and center at their midpoint $X_{1385}$. The locus of $U$ is a second axis-aligned ellipse (dashed light blue) whose center is the incenter $X_1$, and whose co-vertices are a diameter of the caustic. Finally, the locus of $W$ is a circle centered on line $X_1 X_3$.}
    \label{fig:ip-bic}
\end{figure}

\subsection{MacBeath family}

Referring to \cref{fig:ip-mb-locus}, the claims in \cref{obs:bic-combo} are also valid for the MacBeath family. Recall the following fact: the orthocenter $X_4$ of a triangle lies on the directrix of any inscribed parabola \cite{akopyan2007-conics}. As said before, the foci of the MacBeath inellipse are the circumcenter $X_3$ and the orthocenter $X_4$ \cite[MacBeath inellipse]{mw}, and are therefore stationary over the MacBeath family. Therefore:

\begin{corollary}
Over the MacBeath family, the envelope of the directrix of inparabolas with focus a fixed point $F$ on the circumcircle, is the $X_4$-focus of the caustic. 
\end{corollary}

\begin{observation}
Over all $F$, the locus of both $O$ and $U$ are circles. The former is centered on the midpoint $X_{140}$ of the $X_3 X_5$ segment. The latter is concentric with the caustic on $X_5$ and tangent to the caustic at the latter's vertices.
\end{observation}

\begin{figure}
    \centering
    \includegraphics[width=.8\textwidth]{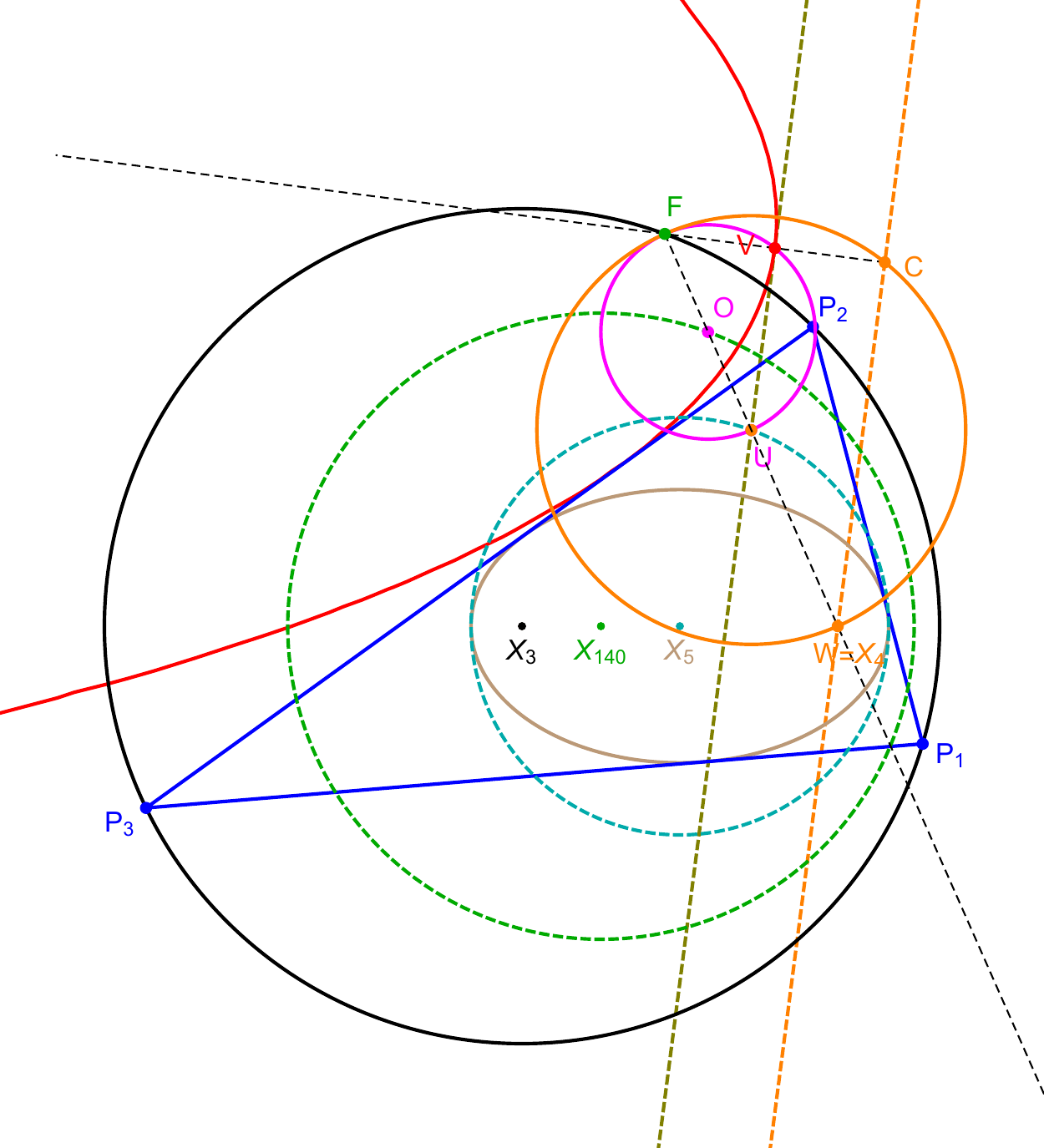}
    \caption{Over the MacBeath family, the locus of $V$ and $C$ are still circles (solid pink and orange, respectively). Since in general the directrix of an inscribed parabola $\P$ passes through $X_4$, we get a powerful stationarity: for any $F$ on the circumcircle, and over any Poncelet triangle, the directrix of $\P$ passes through (the right) focus of the MacBeath caustic, labeled $W=X_4$. All Simson lines (dashed dark green) pass through the antipode $U$ of $F$ on the circular locus of $V$. Over all $F$, the locus of $O$ is a circle (dashed green) centered at the midpoint $X_{140}$ of the $X_2 X_3$, and that of $U$ (the fixed point of the Simson lines) is a circle concentric and internally tangent to the caustic.}
    \label{fig:ip-mb-locus}
\end{figure}

Referring to \cref{fig:ip-mb-polar}:

\begin{observation}
Over the MacBeath family, the locus of the circumcenter $X_3'$ of polar triangles with respect to inparabolas with fixed focus $F$ on the circumcircle is a line. As $F$ sweeps the circumcircle, said linear locus envelops a conic whose major axis coincides with the major axis of the MacBeath inellipse. Said conic has one focus on the center $X_5$ of the MacBeath inconic (caustic).
\end{observation}

\begin{observation}
Over the MacBeath family, the locus of the Brianchon point of inparabolas with fixed focus $F$ on the circumcircle is an ellipse. In general, the locus of the center of said ellipses is not a conic.
\end{observation}

\begin{figure}
    \centering
    \includegraphics[trim=0 0 0 0,clip,width=.8\textwidth]{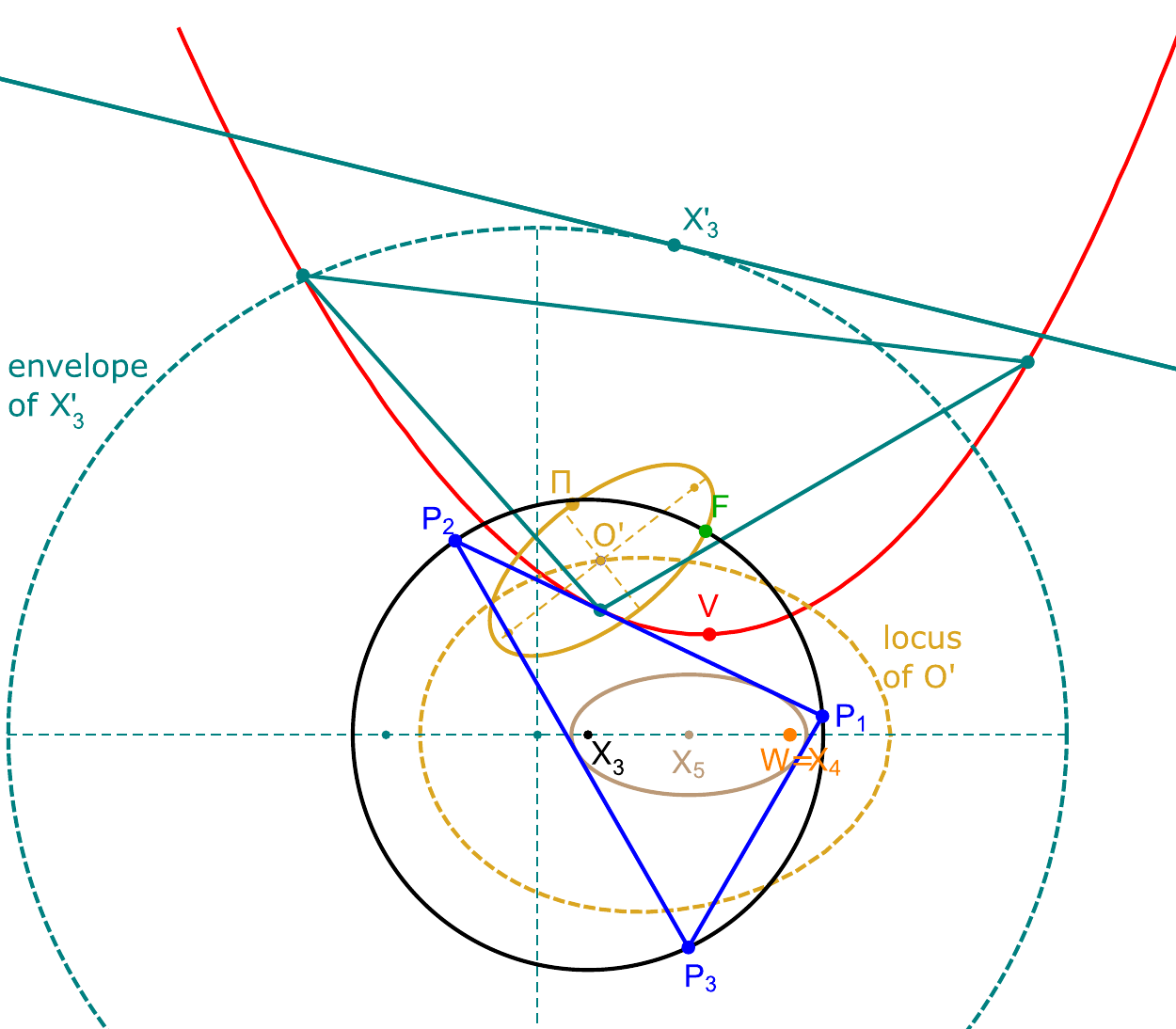}
    \caption{Over the MacBeath family, the locus of the Brianchon point \cite{mw} $\Pi$ of inparabolas $\P$ with fixed focus $F$ on the circumcircle (black) is a conic (gold). Over all $F$, the locus of their centers $O'$ is an oval (dashed gold). Over said Poncelet family, the locus of the circumcenter $X_3'$ of polar triangles with respect to $\P$ is a line (solid teal). Over all $F$, said lines envelop a conic whose major axis coincides with the caustic's, and with one focus at the center ($X_5$) of the MacBeath caustic.}
    \label{fig:ip-mb-polar}
\end{figure}

\subsection{Brocard family}

Referring to \cref{fig:ip-broc-locus}, the claims in \cref{obs:bic-combo} are also valid for the Brocard family. Furthermore:

\begin{observation}
The locus of point $W$, common to all directrices, is a circle with center collinear with the centers of the two Poncelet conics, i.e., on the $X_3 X_{39}$ line.
\end{observation}

\begin{observation}
Over all $F$ on the circumcircle, the locus of the center $O$ of the (circular) locus of $V$ is an ellipse whose minor axis coincides with that of the Brocard inellipse, centered at the midpoint of $X_3 X_{39}$.
\end{observation}

\begin{observation}
Over all $F$, the locus of $U$ common to all Simson lines is an ellipse axis-aligned and concentric with the Brocard inellipse, to which it is tangent internally at both co-vertices.
\end{observation}

\begin{figure}
    \centering
    \includegraphics[trim=30 50 50 30,clip,width=\textwidth]{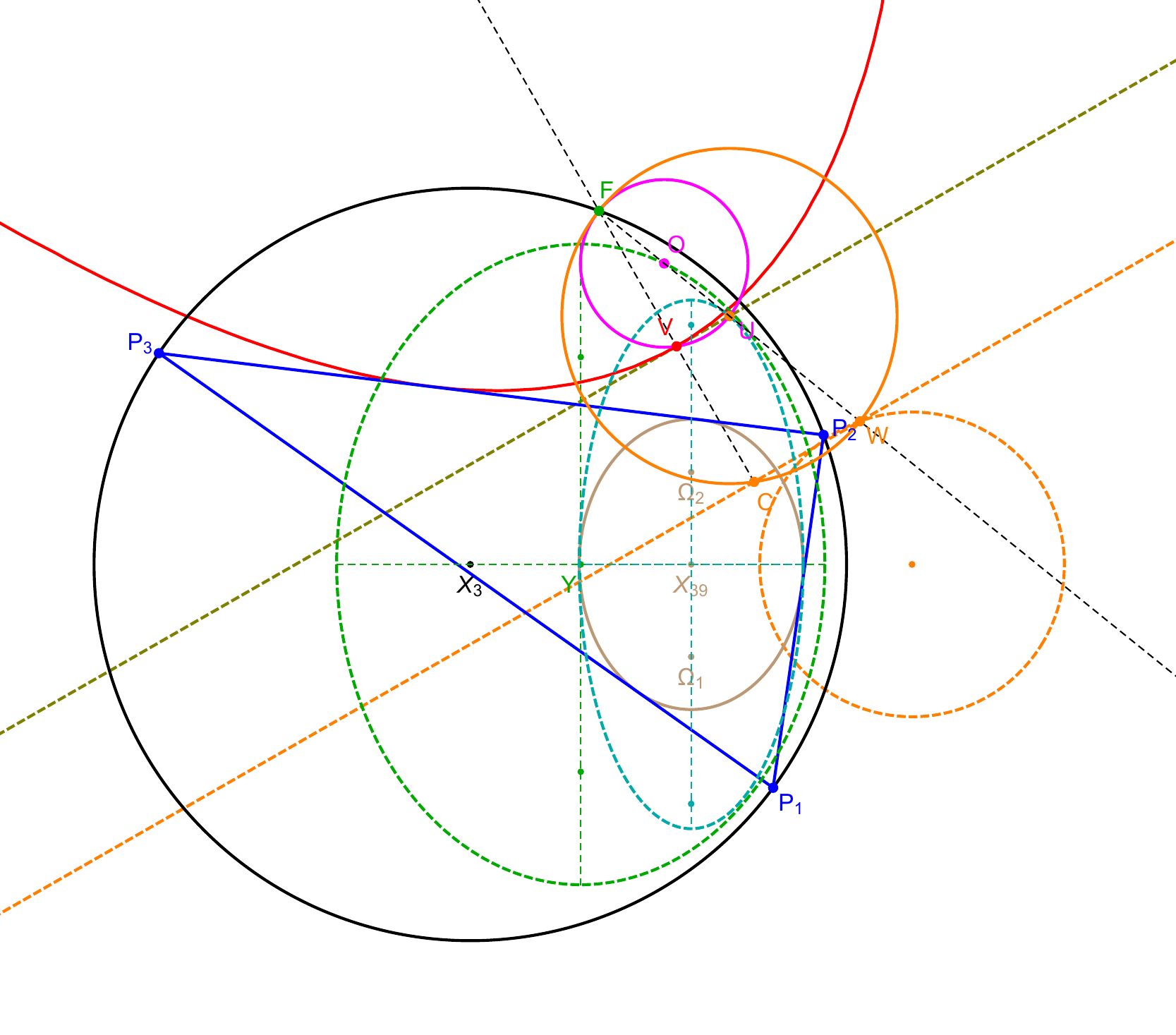}
    \caption{Over the Brocard family, the locus of $V$ and $U$ are again circles (solid pink and orange, respectively). Over all $F$, (i) the locus of the fixed point $W$ (envelope of the directrix) is a circle with center on the minor axis of the caustic $\E'$ (known as the Brocard inellipse \cite{mw}); (ii) the locus of $O$ is an ellipse (dashed green) whose minor axis coincides with that of $\E'$ and whose center is the midpoint $Y$ of $X_3$ and $X_{39}$;  (iii) the locus of $U$ is an ellipse axis-aligned and concentric with $\E'$, and tangent to the latter at both co-vertices.}
    \label{fig:ip-broc-locus}
\end{figure}

Referring to \cref{fig:ip-broc-polar}:

\begin{observation}
Over the Brocard family, the locus of the Brianchon point $\Pi$ of inparabolas with fixed focus $F$ on the circumcircle is an circle. Over all $F$, the locus of the center of this circle is a conic whose major axis is along the $X_3 X_{39}$ line.
\end{observation}

\begin{figure}
    \centering
    \includegraphics[trim=200 10 0 0,clip,width=.8\textwidth]{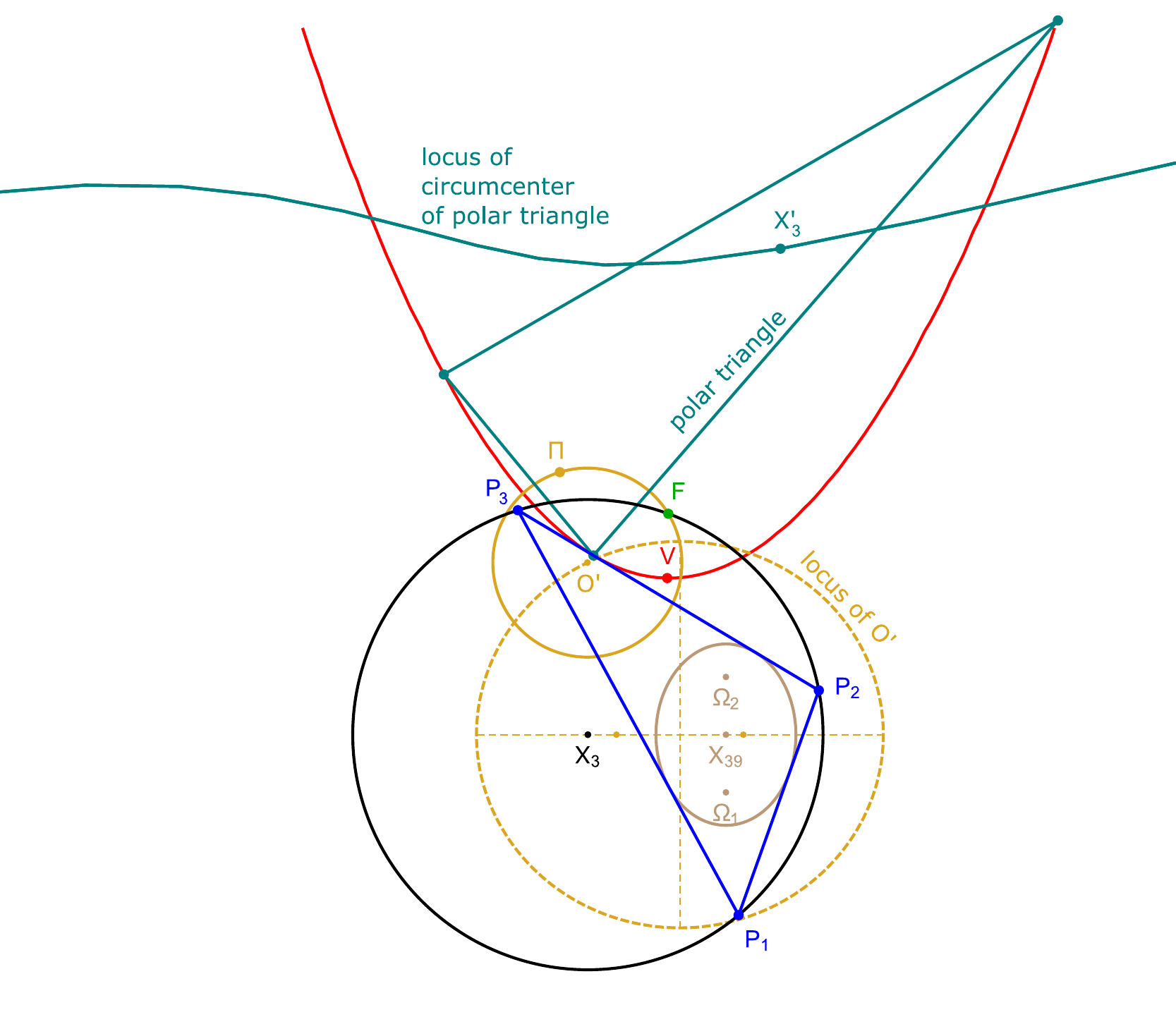}
    \caption{Over the Brocard family, the locus of the circumcenter $X_3'$ of the polar triangle (teal) with respect to inparabolas (red) with fixed focus $F$ is a sinuous curve (teal). The locus of the Brianchon point $\Pi$ is a circle (gold). Interestingly, over all $F$, the locus of the center $O'$ of the locus of $\Pi$ is a conic with major axis along the $X_3 X_{39}$ line.}
    \label{fig:ip-broc-polar}
\end{figure}

\subsection{General circle-inscribed Poncelet}

Consider a circle-inscribed Poncelet triangle family where the inner conic is some generic nested ellipse. Let $F$ be a fixed point on the circumcircle. As mentioned above, the Simson line with respect to $F$ is tangent to the inparabola with focus on $F$ at its vertex $V$. So $V$ can be regarded as the perpendicular projection of $F$ onto the Simson line \cite{gheorghe2021-private}. 


Referring to \cref{fig:ip-ctr-locus}:

\begin{figure}
    \centering
    \includegraphics[trim=0 50 0 0,clip,width=.8\textwidth]{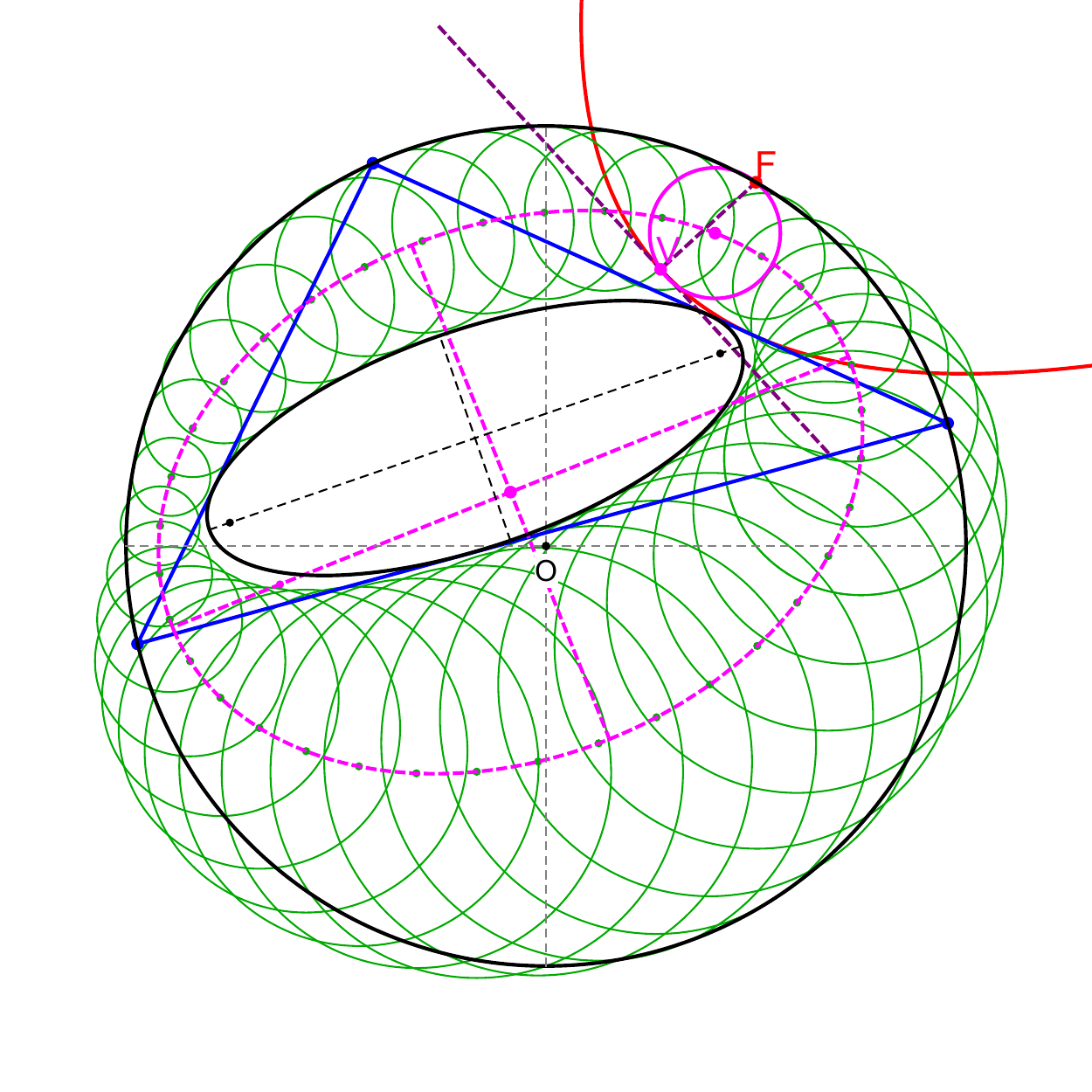}
    \caption{Consider a circle-inscribed Poncelet triangle family (blue) with a caustic/inconic in general position. The locus of the vertex $V$ of inparabolas with focus at a fixed point $F$ on the circumcircle is still a circle (magenta). Over all $F$, the center of said locus (green) sweeps an ellipse (dashed magenta), neither concentric nor axis-aligned with either Poncelet conics.}
    \label{fig:ip-ctr-locus}
\end{figure}

\begin{theorem}
Over an arbitrary Poncelet triangle family inscribed in a circle, the locus of the perpendicular projection of $F$ onto the Simson line is a circle.
\label{conj:ip-circum-vtx}
\end{theorem}

The following proof was kindly provided by Alexey Zaslavsky \cite{zaslavsky2021-private}.

\begin{proof} A sketch is the following.
 Identify the circumcircle with the unit circle in the complex plane. Let $f_1$, $f_2$ be the complex numbers corresponding to the foci of the inconic, and set $F=1$. Let $a,b,c$ denote the sidelengths. Then we have $a+b+c=f_1+f_2+\overline{f_1f_2}abc$ and $ab+bc+ca=f_1f_2+(\overline{f_1+f_2})abc$. The projection of $F$ onto $AB$ is $(1+a+b-ab)/2$; that onto $BC$ and $CA$ are obtained cyclically. From this obtain that the projection $V$ of $F$ onto the Simson line is $V=(1+k-\overline{k}abc)/2$, where $k=f_1+f_2-f_1f_2$, i.e., this point moves along a circle.
\end{proof}
\begin{remark}
A systematic use of complex numbers in planar geometry and  in  Poncelet families of triangles  can be found in   \cite{dsr_applet_x12345}.
\end{remark}
\begin{proposition}
The locus of the isogonal conjugate $V'$ of $V$ is a line tangent to the circumcircle at the antipode of $F$.
\end{proposition}

\begin{proof}
Let $V'$ denote the isogonal conjugate of $V$. This satisfies $V+V'+\bar{VV'}abc=a+b+c$, and we can see that $V'+\bar{V'}=-2$.
\end{proof}

Referring to \cref{fig:ip-locus-W}:

\begin{figure}
    \centering
    \includegraphics[width=.8\textwidth]{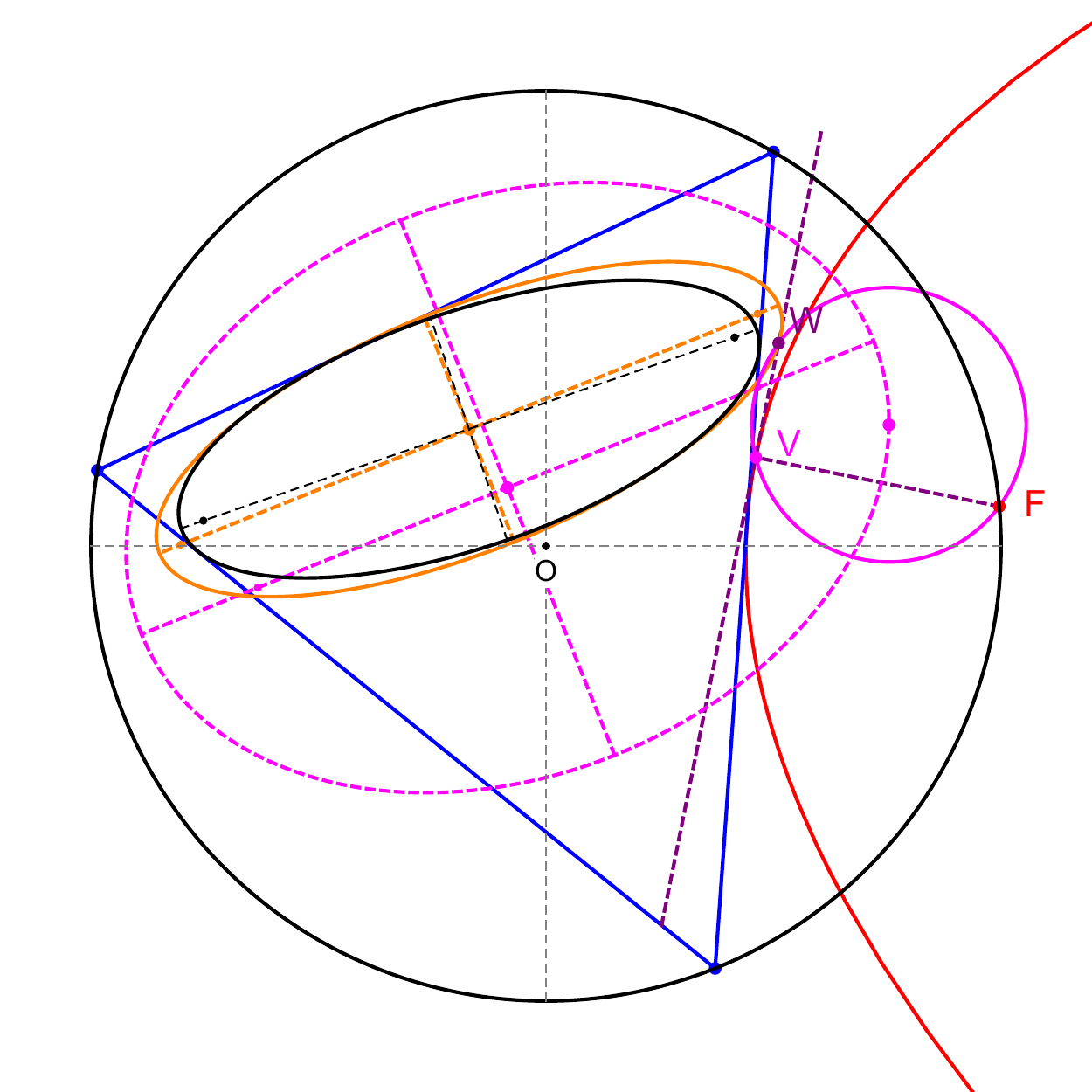}
    \caption{Over circle-inscribed Poncelet with a generic inconic, the envelope of Simson lines (tangents to inparabolas at $V$) is a point $W$ which is the antipode of $F$ on the circular locus of $V$. Over all $F$, $W$ sweeps an ellipse (orange) concentric though not-axis aligned with the inconic.}
    \label{fig:ip-locus-W}
\end{figure}

\begin{observation}
Over any Poncelet triangle family inscribed in a circle, the envelope of Simson lines (dashed purple) is a point $W$ antipodal to $F$ on the circular locus of $V$. \label{conj:ip-W}
\end{observation}

\begin{conjecture}
Over all $F$, the locus of $W$ is an ellipse concentric with the inconic/caustic. \label{conj:ip-locus-W}
\end{conjecture}  
 \section{Inparabolas over Steiner-Inscribed Poncelet}
\label{sec:ip-steiner-inscribed}

A well-known fact is that while the focus to inparabolas lie on the circumcircle, the Brianchon point must lie on the Steiner ellipse \cite[Brianchon point]{mw}. Let $\Pi$ be a fixed point on the outer ellipse of the homothetic Poncelet triangle family.

Referring to \cref{fig:ip-steiner-loci},
let $\Pi$ be a fixed point on the outer (Steiner) ellipse of the ``homothetic'' family. Let $\P$ be an inparabola (red) whose Brianchon point is $\Pi$.
\begin{figure}
    \centering
    \includegraphics[trim=50 30 0 0,clip,width=0.8\textwidth]{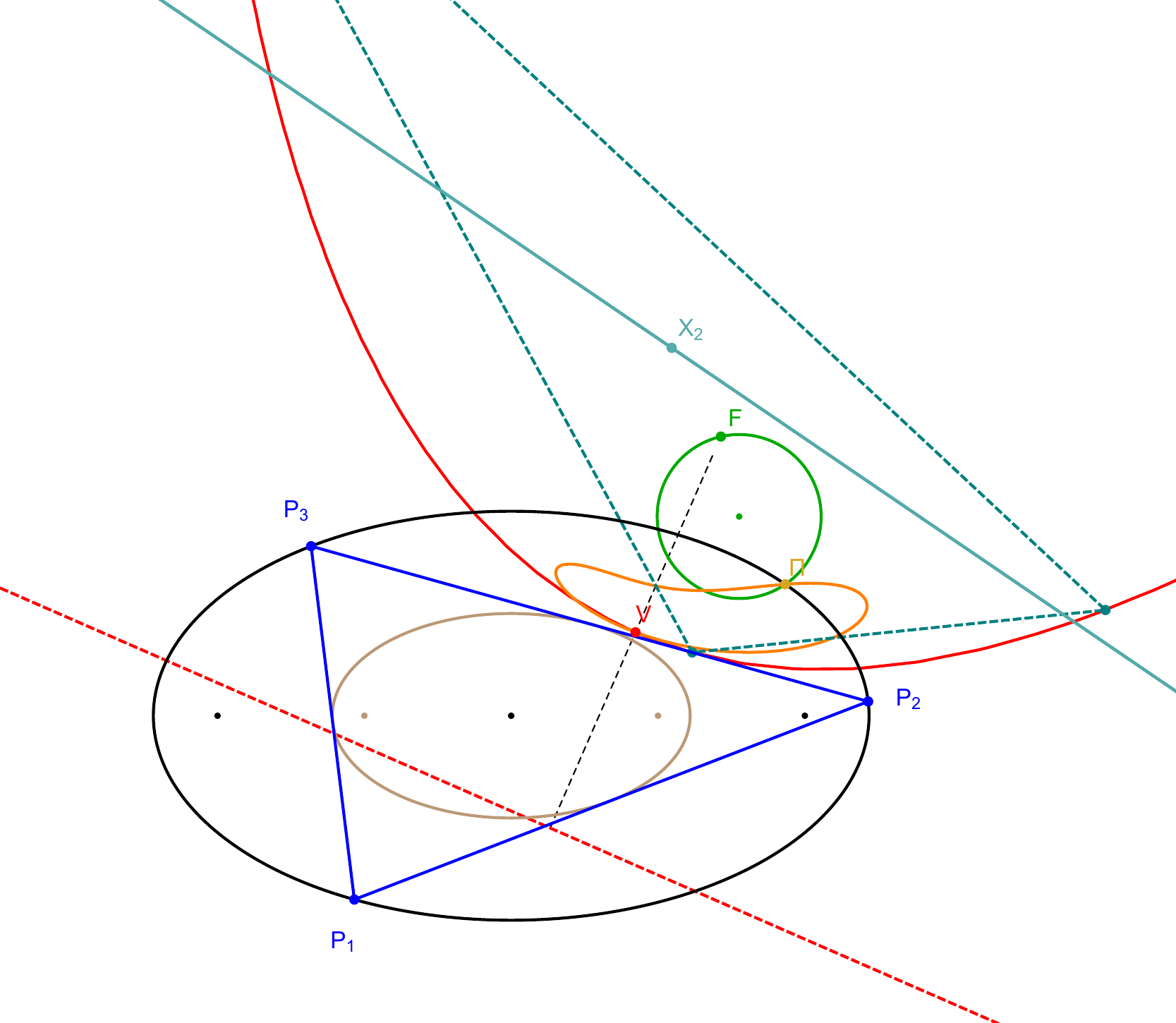}
    \caption{ Over the Poncelet family, the locus of the focus of $\P$ is a circle (green), while the vertex sweeps a non-conic curve (orange). Interestingly, the locus of the barycenter $X_2$ of the polar triangle (dashed teal) is a straight line (solid cyan).}
    \label{fig:ip-steiner-loci}
\end{figure}
\begin{observation}
Over the homothetic family, the locus of the foci of inparabolas whose Brianchon point is a fixed point $\Pi$ on the outer ellipse is a circle.
\end{observation}

Interestingly:

\begin{observation}
Over the homothetic family, the locus of the barycenter of polar triangles with respect to inparabolas with fixed Brianchon point $\Pi$ on the outer ellipse is a circle or a line.
\end{observation}

We suggest: 


\begin{challenge}
Describe the envelope of the directrix (and/or Simson line) over the homothetic family with a fixed $\Pi$ on the Steiner ellipse. 
\end{challenge}

\begin{challenge}
Describe the locus of the \underline{center} of the focus locus over all $\Pi$ on the Steiner ellipse.
\end{challenge} 
\section{Circumparabolas as isogonal images}
\label{sec:cp-isog}

In this section we consider circumparabolas which are isogonal images of a fixed line tangent to the circumcircle. We call these ``isogonal CPs'' for short.

Specifically, below we mention properties of such parabolas over certain Poncelet triangle families inscribed in a circle $\C$ and circumscribing an inner ellipse $\E'$. Let $R$ denote the radius of the outer circle.



\subsection{Focus Locus Hocus Pocus}

For a fixed triangle, the locus of the focus over all possible circumparabolas is a
quintic $Q077$. The geometric construction of this curve  can be found in \cite{gibert2021-q077}.   \cref{remark:Q077} in \cref{app:explicit}.
However, here triangles are Poncelet-varying. Referring to \cref{fig:cp-isog-bic-linear-focus}, the following phenomenon is proved in \cite{odehnal2022-parabolas}:

\begin{theorem}
Over the bicentric family, the locus of the focus of isogonal circumparabolas is a straight line.
\label{thm:focus}
\end{theorem}

\begin{observation}
Over the bicentric family, the locus of the barycenter $X_2'$ of the polar triangle with respect to isogonal circumparabolas is a straight line parallel to the locus of the focus.
\end{observation}


\begin{challenge}
Describe the envelope of the linear focus locus over all tangents to the circumcircle (pre-images of a given isogonal CP family).
\end{challenge}

Referring to \cref{fig:cp-isog-bic-linear-focus}, let $T$ be the intersection of the linear focus locus with the fixed tangent to the circumcircle.

\begin{challenge}
Describe the locus of $T$ over all tangents to the circumcircle.
\end{challenge}

\begin{figure}
    \centering
    \includegraphics[trim=40 100 120 0,clip,width=0.8\textwidth]{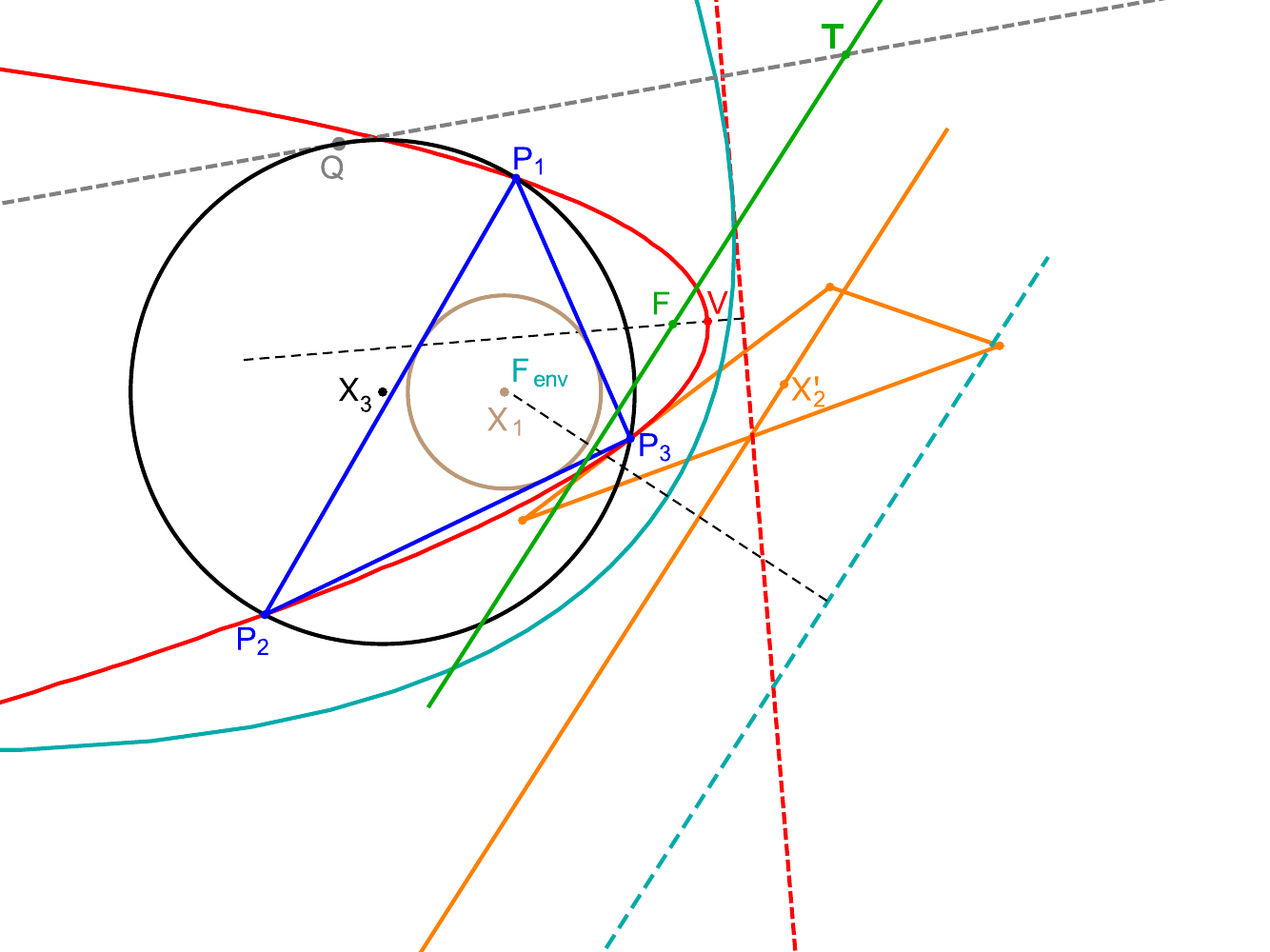}
    \caption{Over the bicentric family, (i) the locus of the focus of isogonal circumparabolas (red) is a straight line (green). Also a straight line is (ii) the locus of the barycenter $X_2'$ of the polar triangle (orange) with respect to the circumparabolas. Note that (i) and (ii) are parallel. (iii) the envelope of the directrix (dashed red) is a parabola (cyan) with focus $F_{env}$ at the incenter $X_1$ and directrix (dashed cyan) parallel to (i) and (ii). Point $T$ is the intersection of the linear focus locus with the fixed tangent to the circumcircle.}
    \label{fig:cp-isog-bic-linear-focus}
\end{figure}

\subsection{Directrix Envelope}

Referring to \cref{fig:cp-isog-bic-linear-focus}:

\begin{observation}
Over bicentric family, the envelope of the directrix of isogonal circumparabolas is a parabola with focus on the center $X_1$ of the inscribed circle. Furthermore, the directrix of this parabolic envelope is parallel to the loci of $F$ and $X_2'$. 
\end{observation}

Referring to \cref{fig:cp-isog-inell}, over the inellipse family, neither the locus of the focus nor that of the vertex are low degree curves, however:

\begin{figure}
    \centering
    \includegraphics[trim=20 20 100 20 ,clip,width=.8\textwidth]{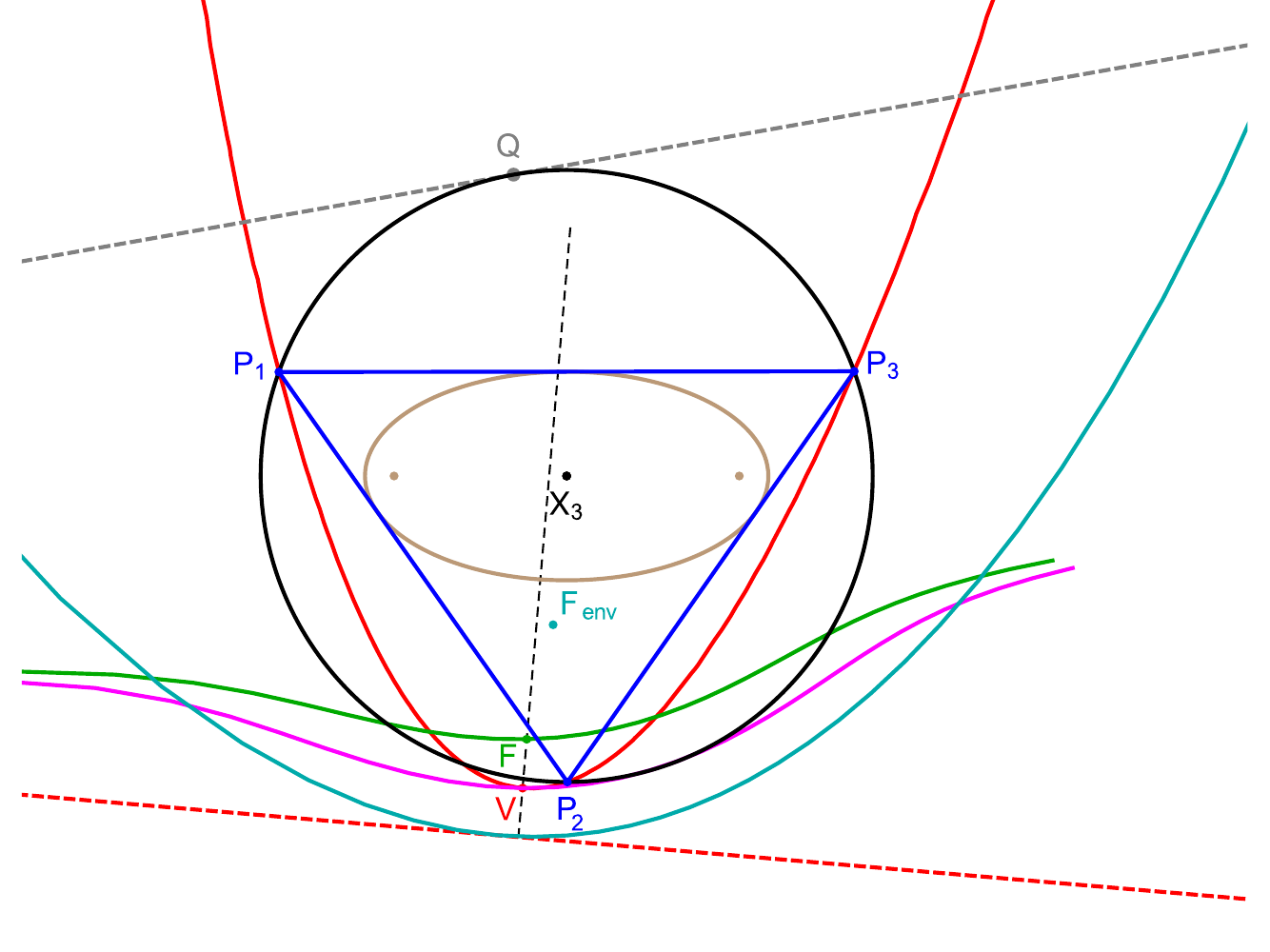}
    \caption{Over the ``inellipse'' family, the locus of the focus and vertex of isogonal circumparabolas (red) are curves of degree higher than 2 (red and magenta, respectively). The directrix's (dashed red) envelope (cyan) is  a parabola (cyan). Its focus is shown as $F_{env}$.}
    \label{fig:cp-isog-inell}
\end{figure}

\begin{observation}
Over the inellipse family, the envelope of the directrix of isogonal circumparabolas is a parabola.
\end{observation}

In fact:

\begin{observation}
Over both the MacBeath and Brocard families, the envelope of the directrix of isogonal circumparabolas are parabolas.
\end{observation}

In turn, this gives credence to:

\begin{conjecture}
Over any Poncelet triangle family inscribed in a circle, the envelope of directrix of isogonal circumparabolas is a parabola.
\label{conj:cp-isog-directrix}
\end{conjecture}

\begin{challenge}
For each circle-inscribed family (other than the bicentric one), describe the locus of the focus of the parabolic directrix envelope over all tangents to the circumcircle which are isogonal pre-images of circumparabolas.
\end{challenge}

\subsection{Perspectors}

Let $\C$ be a circumconic of a triangle $T$. The {\em polar triangle} $T'$ with respect to $\C$ is bounded by the tangents to $\C$ at the vertices of $T$ \cite[Polar triangle]{mw}. The {\em perspector} $\Pi$ of $\C$ is the point at which $T$ and $T'$ are in perspective \cite{mw}.  It is known that the perspectors of all circumparabolas to a fixed triangle sweep the Steiner inellipse \cite{pamfilos2021-circumparabolas}. Referring to \cref{fig:cp-isog-persp}:

\begin{observation}
Over both the bicentric and MacBeath families, the locus of the perspector of isogonal circumparabolas is an ellipse.
\end{observation}

\begin{figure}
    \centering
    \includegraphics[width=\textwidth]{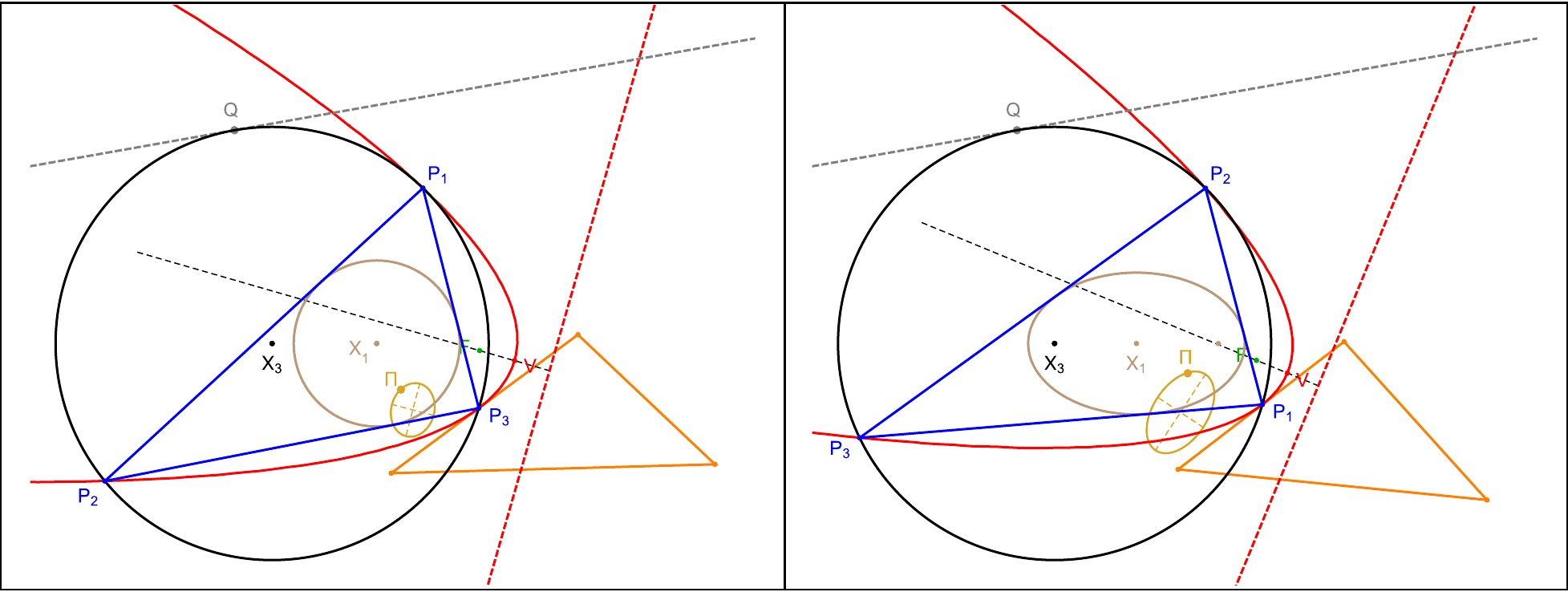}
    \caption{Over both the bicentric (left) and MacBeath (right) families, the locus of the perspector $\Pi$ of isogonal circumparabolas (red) is an ellipse (gold).}
    \label{fig:cp-isog-persp}
\end{figure}

Referring to \cref{fig:cp-isog-persp-broc}:

\begin{observation}
Over the Brocard family, the locus of the perspector of isogonal circumparabolas is a circle.
\end{observation}

\begin{figure}
    \centering
    \includegraphics[trim=0 0 50 0,clip,width=.8\textwidth]{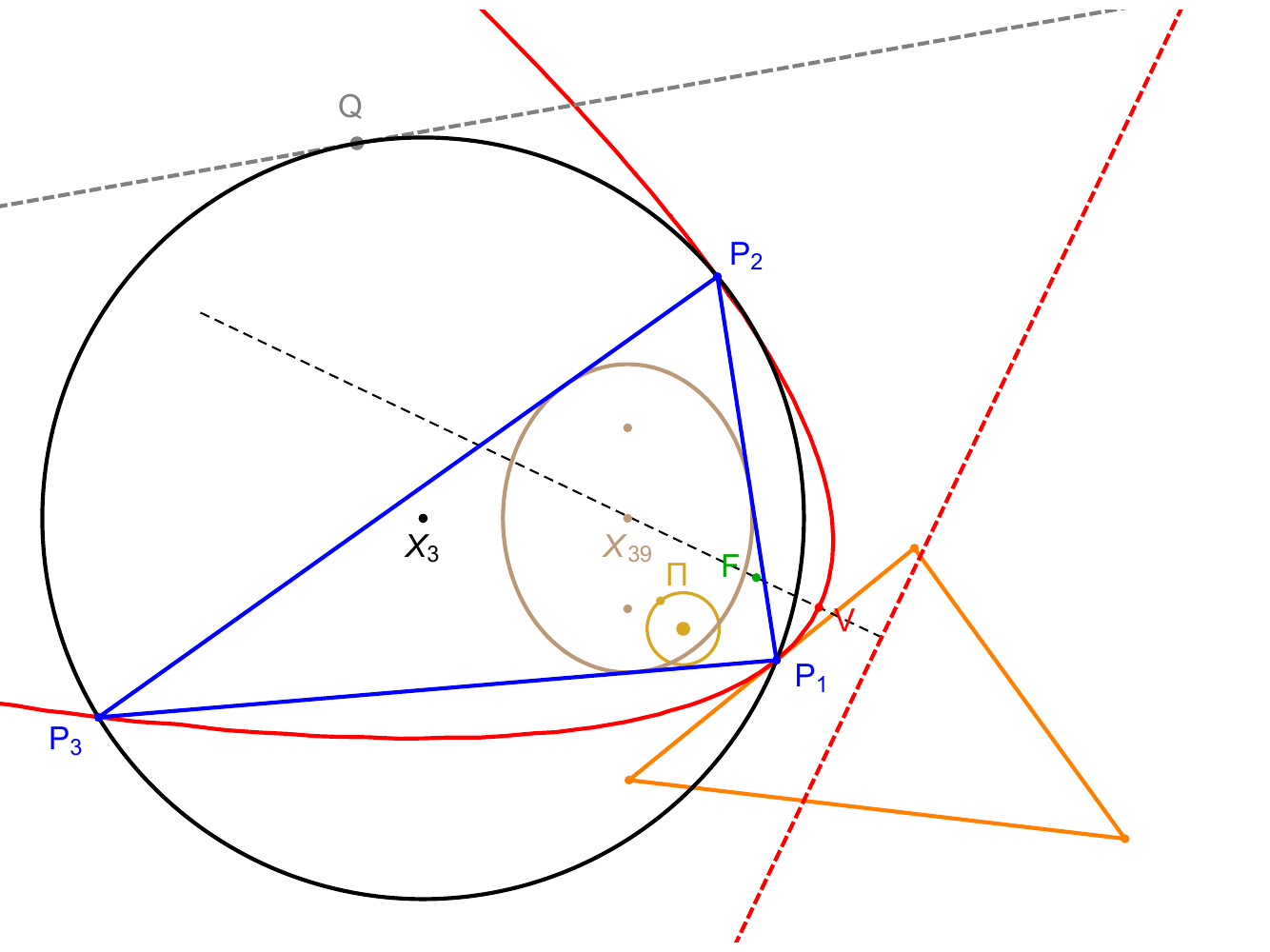}
    \caption{Over the Brocard family, the locus of the perspector $\Pi$ of isogonal circumparabolas (red) is a circle (gold).}
    \label{fig:cp-isog-persp-broc}
\end{figure}

Let $\Pi_Q$ denote the locus of the perspector of circumparabolas isogonal to a line tangent to the circumcircle at $Q$.

\begin{challenge}
Over all $Q$, describe the locus of the \underline{center} of $\Pi_Q$ generated over bicentric, MacBeath, and Brocard families.
\end{challenge} 
\section{Circumparabolas as isotomic images}
\label{sec:cp-isot}

In this section we consider circumparabolas which are  isotomic images of a fixed line $\L$ tangent to the Steiner (circum)ellipse. We call these ``isotomic CPs'' for short. Below we enumerate some salient properties of such parabolas over a family of Poncelet triangles interscribed between two homothetic ellipses $\E$ and $\E'$, see \cref{fig:poncelet-tris}(right). Recall these are precisely the Steiner circum- and inellipse, respectively, centered at the barycenter $X_2$ of a general triangle. Since this family is the affine image of equilaterals interscribed between two concentric circles, it conserves area and maintains the affinely-invariant barycenter\footnote{The barycenter is the sole triangle center invariant under affine transformations.} $X_2$ stationary. Indeed, it conserves a myriad of other quantities such as sum of squared sidelengths, Brocard angle, etc. \cite{garcia2020-family-ties}.

Referring to \cref{fig:cp-isot-tangency}, the following has been kindly proved by B. Gibert \cite{gibert2021-private}. Let $\E$ and $\E'$ denote the outer and inner ellipse in the homothetic pair.

\begin{proposition}
Over the homothetic family, all isotomic circumparabolas are tangent to the reflection of $\L$ with respect to the common center $X_2$. Said circumparabolas envelop an ellipse which is axis-parallel with $\E,\E'$ and is tangent to $\E$ at $Q$ and to $\E'$ at $Q'$ where $Q$ is where $\L$ touches $\E$ and $Q'$ is the intersection of $Q X_2$ with $\E'$ farthest from $Q$.
\end{proposition}


\begin{figure}
    \centering
    \includegraphics[trim=0 30 40 40,clip,width=\textwidth]{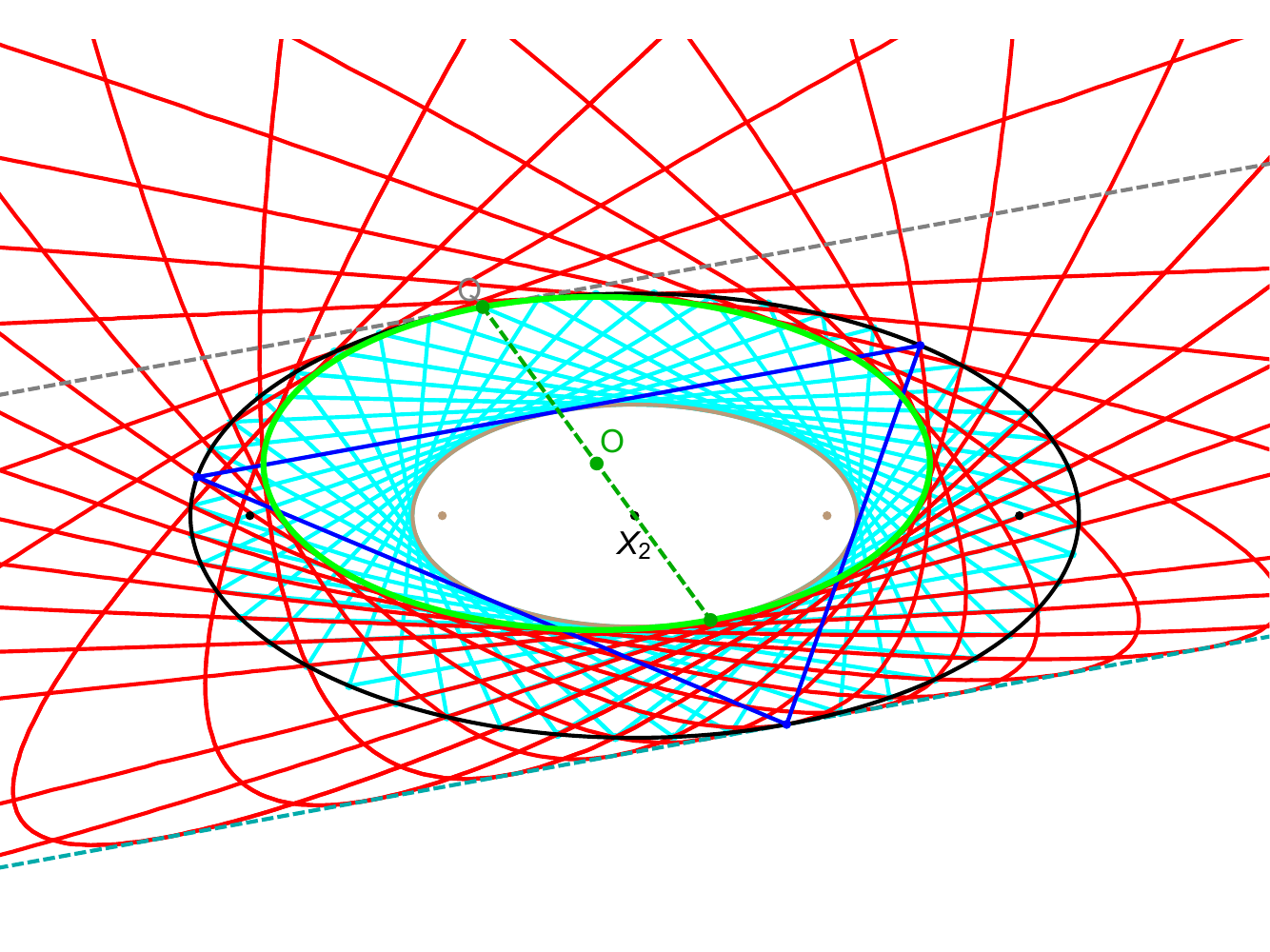}
    \caption{Over the homothetic family, all isotomic circumparabolas are tangent to the reflection of the tangent line $\L$ with respect to the common center $X_2$. This family of circumparabolas envelops an ellipse (green) axis-aligned with the homothetic pair and with center at the midpoint of Q and the distal intersection of line $QX_2$ with the caustic.}
    \label{fig:cp-isot-tangency}
\end{figure}
Referring to \cref{fig:cp-isot-loci}, 
   consider a  triangle (blue)  interscribed between two concentric, homothetic ellipses $\E$ and $\E'$ (the Steiner ellipse and inellipse, respectively). Consider the circumparabola $\P$ (red) which is the isotomic image of a line $\L$ tangent to $\E$ at $Q$.
   
One notices that over said family, the locus of either the focus or vertex of isotomic circumparabolas are sinuous curves. However:

\begin{observation}
Over the homothetic family, the envelope of the directrix of isotomic circumparabolas is a parabola. Furthermore, 
the directrix of said envelope is a line parallel to $\L$
\end{observation}

Furthermore:

\begin{observation}
Over the homothetic family, the locus of the barycenter of the polar triangle with respect to isotomic circumparabolas is a line parallel to $\L$.
\end{observation}

\begin{figure}
    \centering
    \includegraphics[trim=50 0 0 0,clip,width=0.8\textwidth]{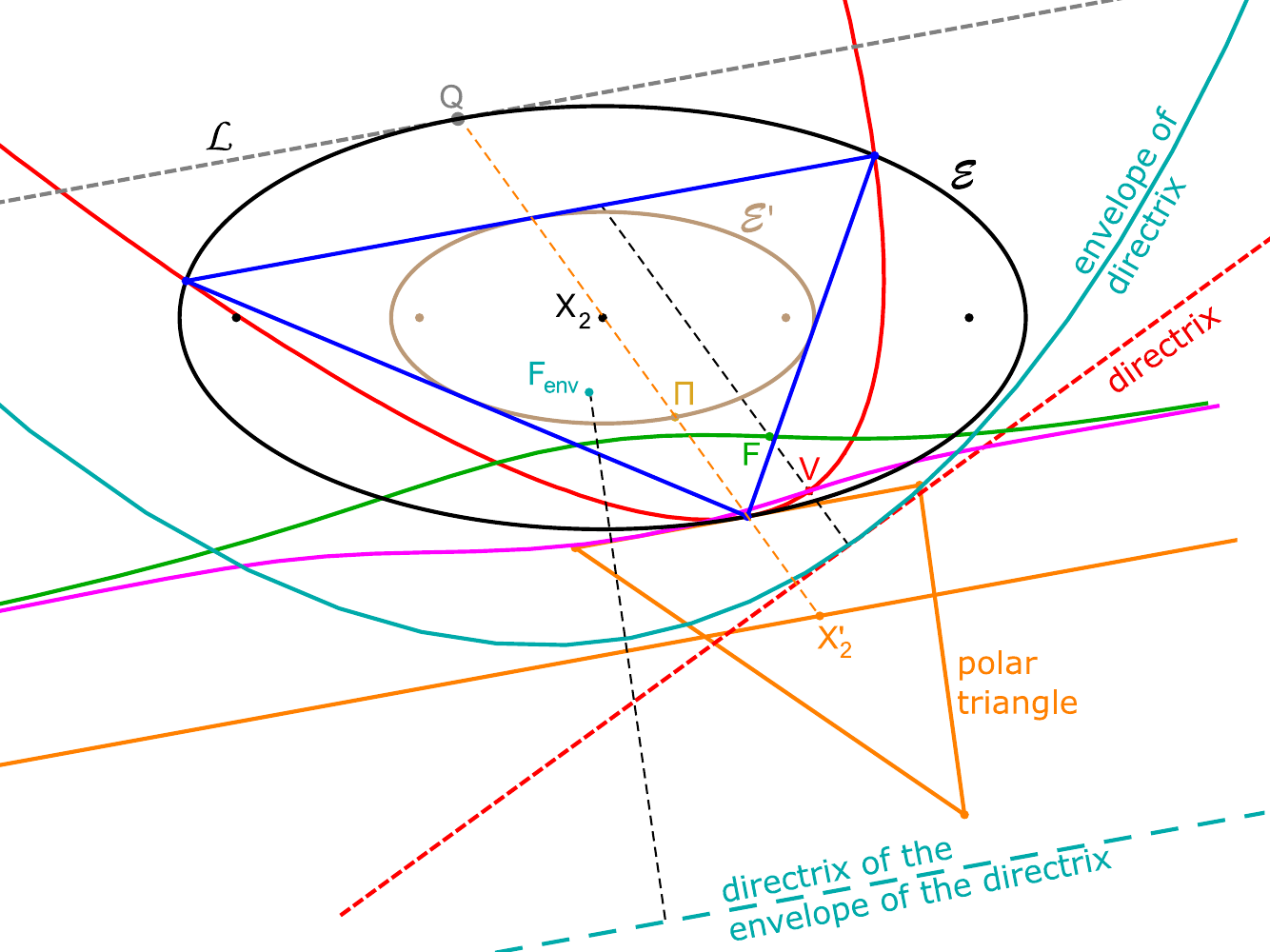}
    \caption{
    Over the Poncelet family, the locus of the focus $F$ and vertex $V$ of the circumparabolas $\P$ are sinuous curves (green and magenta, respectively). Interestingly, the locus of the barycenter $X_2'$ of the polar triangle (orange) with respect to $\P$ is a straight line parallel to $\L$. The envelope of the directrix of $\P$ (dashed red) is a parabola (cyan, $F_{env}$ indicates its focus), whose directrix (dashed cyan) is parallel to $\L$. Remarkably, over the Poncelet family, the perspector $\Pi$ of $\P$ (necessarily on $\E'$ \cite{pamfilos2021-circumparabolas}) remains stationary and is collinear with the tangency point $Q$ of $\L$ and $X_2$.}
    \label{fig:cp-isot-loci}
\end{figure}

\begin{observation}
Over the homothetic family, the perspector $\Pi$ of isotomic circumparabolas is stationary on the Steiner inellipse and collinear with $X_2$ and the touch-point $Q$ of $\L$ on the outer Steiner ellipse.
\end{observation}

\begin{challenge}
Over all tangents to the Steiner ellipse which are pre-images of isotomic circumparabolas, describe the locus of the focus of the parabolic directrix envelope swept over the homothetic family.
\end{challenge}

\subsection{Locus of Generatrix Intersection}

Referring to \cref{fig:cp-tang-int-locus}, consider both the isogonal and isotomic pre-images of some circumparabola of a triangle $T$. As mentioned above, these are lines tangent to the circumcircle and Steiner ellipse, respectively. Let $Z$ denote their intersection, and $Q$ and $R$ denote the tangency points, respectively.

Recall the definition of the {\em Steiner point} $X_{99}$ of a triangle \cite{etc}: it is 4th intersection of the circumcircle with the Steiner ellipse (the first 3 are the vertices).

\begin{observation}
$Q$, $R$, and the Steiner Point $X_{99}$ are collinear.
\end{observation}

The Kiepert parabola \cite{mw} is a special inconic whose directrix is the Euler line\footnote{Called the ``magic highway'' of a triangle in this \href{https://youtu.be/wVH4MS6v23U}{video}, the Euler line passes through the barycenter $X_2$, circumcenter $X_3$, orthocenter $X_4$, 9-pt center $X_5$ and a dozens of other triangles centers \cite{mw}.}. Its focus (necessarily on the circumcircle) is $X_{110}$ in \cite{etc}. Still referring to  \cref{fig:cp-tang-int-locus}, the following has been kindly proved by B. Gibert \cite{gibert2021-private}:

\begin{proposition}
Over the 1d family of circumparabolas to a fixed triangle, the locus of $Z$ is the isogonal image of the Kiepert parabola.
\end{proposition}
   Also, it can be shown that over the family of  circumparabolas of $T$, the locus of $Z$ is a curve (green) which is the isogonal image of the Kiepert parabola \cite{mw} (pink), whose focus is $X_{110}$ and the directrix is the Euler line $X_2 X_3$.
\begin{figure}
    \centering
    \includegraphics[trim=0 50 50 70,clip,width=.8\textwidth]{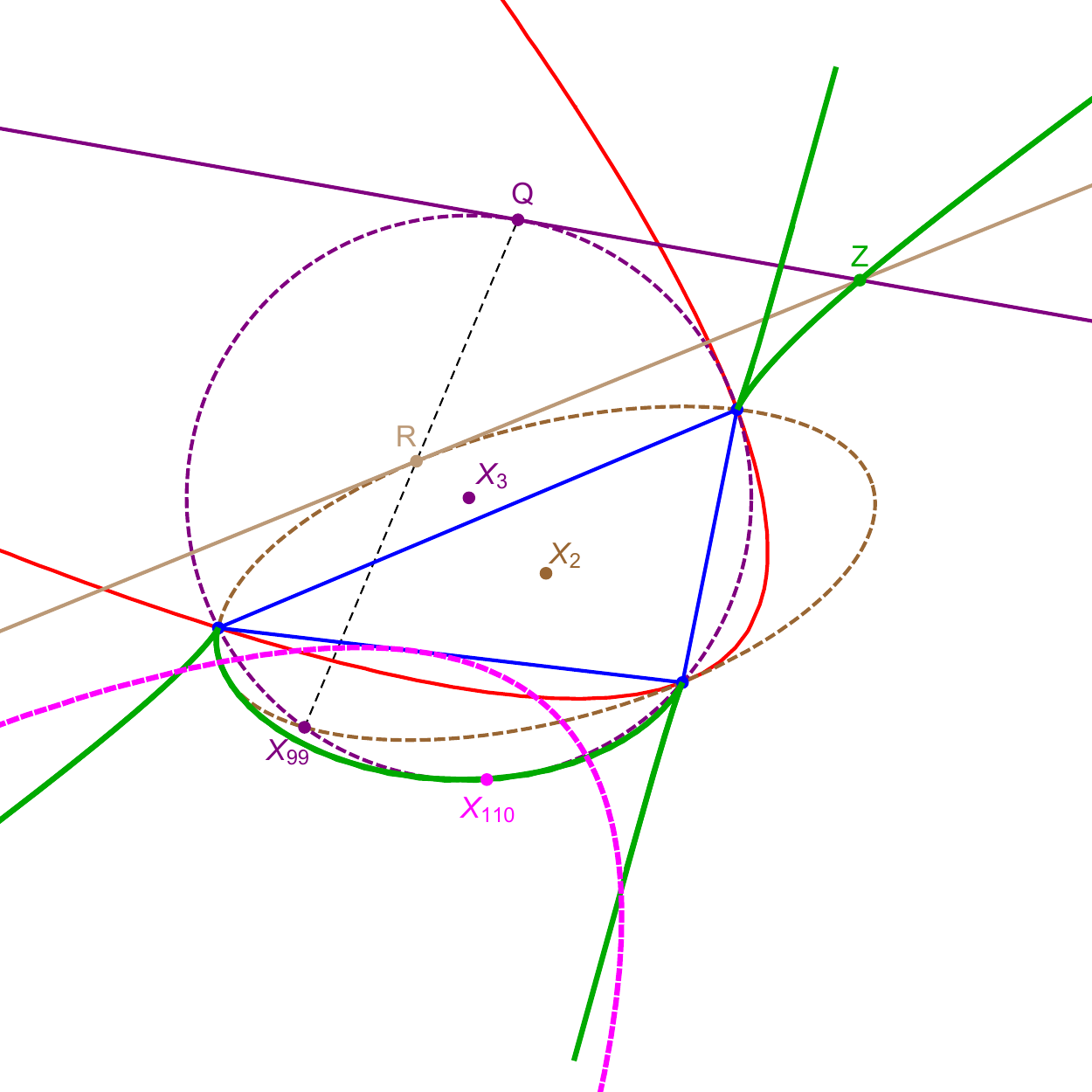}
    \caption{A particular circumparabola (red) is shown of a triangle $T$ (blue). 
    It is the isogonal (resp. isotomic) pre-image of a line tangent at $Q$ to the circumcircle \cite[Circumconic]{mw} (resp. at $R$ to the Steiner ellipse \cite[TC7(2)]{stothers-perspectors}). Let point $Z$ denote their intersection.
    Also shown is the curious fact that $Q$, $R$ and the Steiner point $X_{99}$ are collinear.}
    \label{fig:cp-tang-int-locus}
\end{figure}

\section{Conclusion}
\label{sec:summary}

Narrated videos of some phenomena appear in a YouTube playlist \cite{playlist2021-parabolas}. We invite readers to both contribute proofs and/or work out the challenges proposed above.

\section*{Acknowledgements}
\noindent We would like to thank A. Akopyan, L. Gheorghe, B. Gibert, P. Moses, and A. Zaslavsky for their invaluable insights. We are also grateful to the excellent comments by the referee.

\appendix
\section{Families of Poncelet families}
\label{app:poncelet}


Shown in \cref{fig:app-poncelet-circle-inscribed} are the four circle-inscribed  Poncelet families studied, and defined as follows:

\begin{itemize}
    \item Inellipse: $\E'$ is a concentric ellipse with semi-axes $a,b$. $(\C,\E')$ admit Poncelet triangles if $a+b=R$ \cite{garcia2020-family-ties}.
    \item Bicentric (also known as Chapple's porism): $\E'$ is a circle of radius $r$. Let $d=|OI|=|X_1 X_3|$ denote the distance between fixed incenter and circumcenter. The so-called ``Chapple-Euler'' condition for Poncelet triangle admissibility\footnote{William Chapple published it in 1746 and Leonard Euler in 1765, see this \href{https://en.wikipedia.org/wiki/William_Chapple_(surveyor)}{wikipedia page}.} is that $d^2={R(R-2r)}$. For the historical background, see \cite[Sec.1.1]{centina15}.
    \item MacBeath porism: $\E'$ is the so-called MacBeath inellipse \cite{mw}, whose foci are $X_3$ and $X_4$, and center is $X_5$, the center of the 9-point circle. As shown in \cite{odehnal2011-poristic,pamfilos2020,garcia2020-similarity-I}, this can be regarded as the family of excentral triangles\footnote{The excentral triangle has sides along the external bisectors of a triangle.} of the bicentric family.
    \item Brocard porism: $\E'$ is the Brocard inellipse \cite{mw}, whose foci are the two stationary Brocard points of the family \cite{bradley2007-brocard,reznik2020-similarityII}. These triangles conserve Brocard angle and are also known as the $N=3$ harmonic family \cite{casey1888}.
\end{itemize}
\begin{figure}[H]
    \centering
    \includegraphics[width=\textwidth]{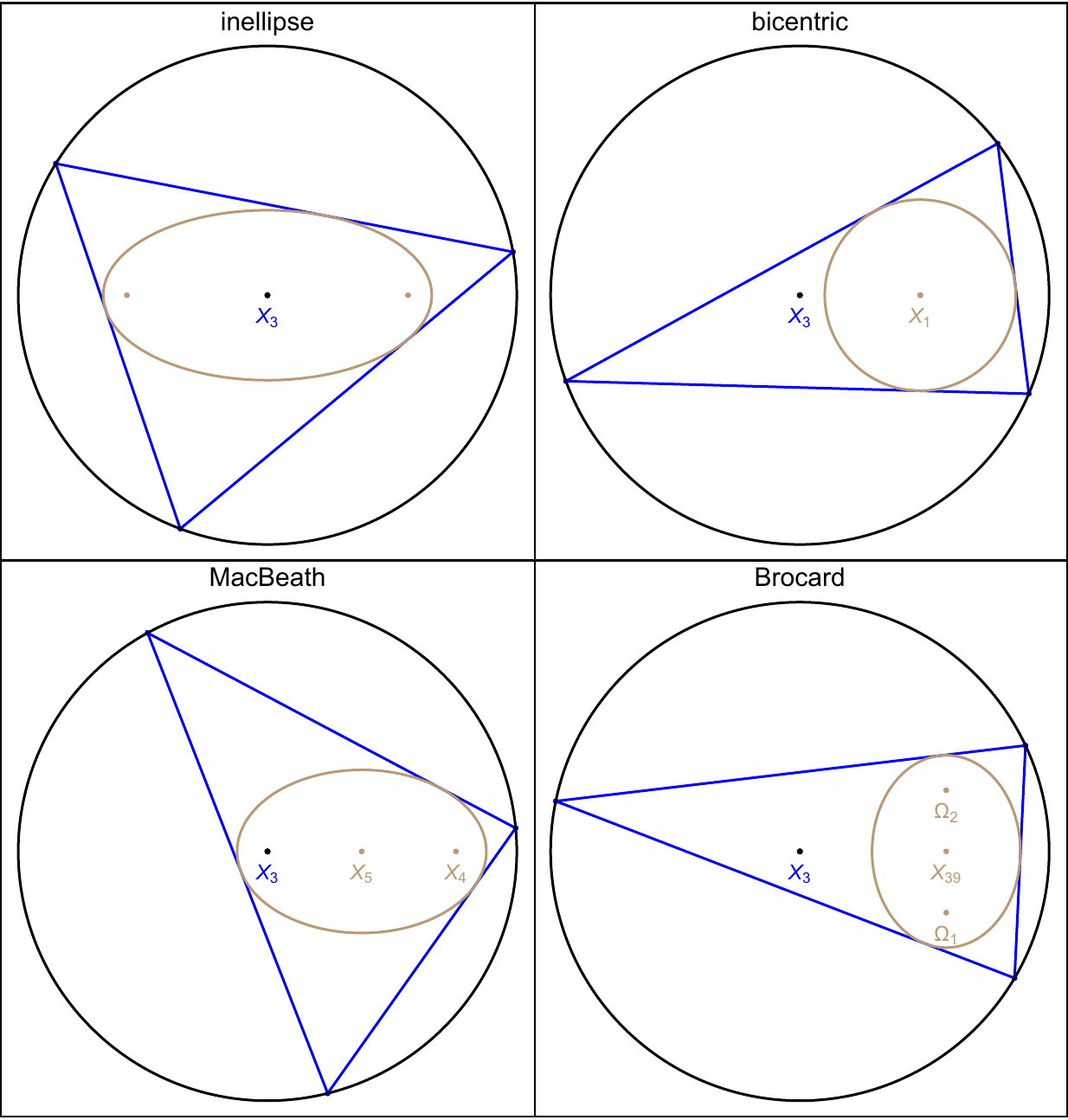}
    \caption{The four circle-inscribed Poncelet triangle families considered herein: (i) ``inellipse'' (caustic is a concentric ellipse), (ii) bicentric, i.e., Chapple's porism, i.e., triangles interscribed between two circles \cite[Sec.1.1]{centina15}; (iii) the ``MacBeath'' family's caustic has one focus on the circumcenter and another one on the orthocenter $X_4$. Its center is that of the 9-point circle $X_5$  \cite[MacBeath inconic]{mw}; (iv) the Brocard porism: the foci of the inconic are the two stationary Brocard points of the family \cite{bradley2007-brocard}.}
    \label{fig:app-poncelet-circle-inscribed}
\end{figure}

\section{Isogonal and Isotomic Conjugation}
\label{app:conjug}
The geometric construction of the isotomic and isogonal conjugate of a point $P$ in the plane of $\triangle ABC$ is illustrated in \cref{fig:conjug}. For more details, see \cite{akopyan12, akopyan2007-conics, garcia2021-impa,  sigur2005-conjug}.

\begin{figure}
    \centering
    \includegraphics[trim=40 140 40 60,clip,width=\textwidth,frame]{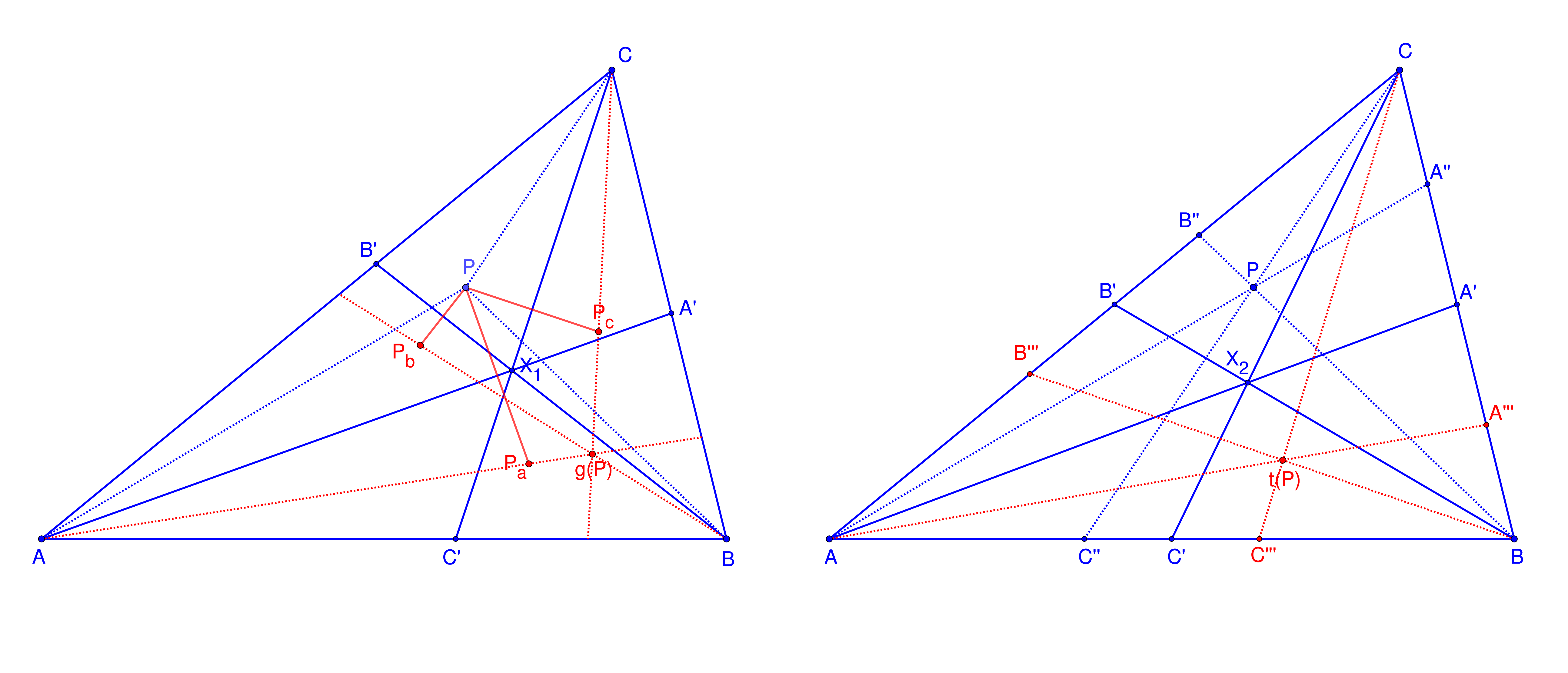}
    \caption{\textbf{Left:} Given a point $P$ in the plane of $\triangle ABC$, let $P_a,P_b,P_c$ be reflections of $P$ about the angle bisectors $AA'$, $BB'$, and $CC'$, respectively. The isogonal conjugate $g(P)$ of $P$ is common intersection of cevians $AP_a$, $AP_b$, and $CP_c$. \textbf{Right:} Let $A'$, $A''$, and $A'''$ denote the (i) midpoint of side $BC$, (ii) the intersection of cevian $AP$ with $BC$, and (iii) the reflection of $A''$ about $A'$ (obtain other points cyclically). 
    The isotomic conjugate $t(P)$ of $P$ is the common intersection of cevians $AA'''$, $BB'''$ and $CC'''$.}
    \label{fig:conjug}
\end{figure}

If the barycentric coordinates of $P$ be $[u,v,w]$, those of $t(P)$ will be $[1/u,1/v,1/w]$ \cite[Isotomic conjugate]{mw}. Likewise, is the trilinear coordinates of $P$ be $[r,s,t]$, those of $g(P)$ will be $[1/r,1/s,1/t]$ \cite[Isogonal conjugate]{mw}.

Recall trilinear and barycentric coordinates are homogeneous triples, i.e., all multiples correspond to the same projective point. Recall that if a point has trilinear coordinates $[r,s,t]$, its barycentric coordinates are $[a r,b s, c t]$, where $a,b,c$ are the sidelengths, i.e., one system is easily converted into the other. For example, the trilinear coordinates of the incenter $X_1$ (resp. barycenter $X_2$) are $[1,1,1]$ (resp. $[b c,c a,a b]$. Its barycentrics coordinates are therefore $[a,b,c]$ (resp. $[1,1,1]$) \cite{etc}.

\section{Explicit Derivations}
\label{app:explicit}
\subsection{Circumparabolas}

The focus $F=(x_f,y_f)$ and directrix $\D$ of a parabola parametrized by:
\[x= a_0+a_1t+a_2 t^2, \;\;y=b_0+b_1 t+b_2t^2\]
are given by:
\begin{align*}
x_f=& \frac{ 4 a_0 (a_2^2 + b_2^2) - a_1 (a_1 a_2 + b_1 b_2) - b_1 (a_1 b_2 - a_2 b_1)}{ 4(a_2^2 + b_2^2) }\\
y_f=&\frac{ 4 b_0 (a_2^2 + b_2^2) + a_1 (a_1 b_2 - a_2 b_1) - b_1 (a_1 a_2 + b_1 b_2) }{4(a_2^2 + b_2^2) } \\
   \D: & \; a_2 x + b_2 y - (a_0 a_2 + b_0 b_2) + \frac{1}{4}(a_1^2 + b_1^2)=0
\end{align*} 

\begin{remark}
Consider a triangle $\T:~P_i=(\cos\alpha_i,\sin\alpha_i), i=1,2,3$, inscribed in the unit circle $\C:~x^2+y^2=1$.
Let $\ell_\theta$ be the line through $P=(\cos\theta, \sin\theta)$ and tangent  to $\C$. The isogonal image of $\ell_\theta$ with respect to $\T$ is the circumparabola given by:
\[x= a_0+a_1t+a_2 t^2, \;\;y=b_0+b_1 t+b_2t^2\]
where:
\begin{align*}
    a_0&= \cos(  \alpha_1 +  \alpha_2 +  \alpha_3-2  \theta)\\
    a_1&=sin( \theta)  - sin( \alpha_2+ \alpha_3 -  \theta ) - sin( \alpha_1+   \alpha_2-  \theta  ) - sin( \alpha_1+\alpha_3 -  \theta)\\
    &+ 2 sin(  \alpha_1 +  \alpha_2 +  \alpha_3-2  \theta  )\\
    a_2&=  \cos( \theta) -  \cos( \alpha_1) -  \cos( \alpha_2) -  \cos( \alpha_3) +  \cos(   \alpha_2+ \alpha_3 -  \theta ) +  \cos(   \alpha_1+\alpha_3 -  \theta ) \\
    &+  \cos(   \alpha_1+ \alpha_2 -  \theta ) - \cos(  \alpha_1 +  \alpha_2 +  \alpha_3-2  \theta ) \\
     b_0&= \sin(  \alpha_1 +  \alpha_2 +  \alpha_3-2  \theta)\\
     b_1&= - \cos( \theta)+   \cos( \alpha_1+\alpha_2 -  \theta) +  \cos( \alpha_1+\alpha_3 -  \theta) +  \cos(\alpha_2+ \alpha_3 -  \theta)\\
     & -   2\cos(  \alpha_1 +  \alpha_2 +  \alpha_3-2  \theta )\\
     b_2&= sin( \theta)-sin( \alpha_1)  - sin( \alpha_2) - sin( \alpha_3) + sin(\alpha_1+ \alpha_3 -  \theta) + sin(\alpha_1+ \alpha_2 -  \theta) \\
     &+ sin( \alpha_2+   \alpha_3-  \theta)  
      - sin( \alpha_1 +  \alpha_2 +  \alpha_3-2  \theta)
\end{align*}
\label{prop:circumparabola}
\end{remark}

\begin{remark}
The envelope of the directrix over the circumparabolas given in \cref{prop:circumparabola} is a rational parametric curve $(x_e(t),y_e(t))$, where $\cos(\theta)=\frac{1-t^2}{1+t^2}, \; \sin(\theta)=\frac{2t}{1+t^2}$. In the implicit form it is given by  a sextic polynomial  equation.
\end{remark}

\begin{remark}
The locus of the focus of circumparabolas over all $\ell_\theta$ in \cref{prop:circumparabola} is a rational parametric curve $(x_f(t),y_f(t))$  where $\cos(\theta)=\frac{1-t^2}{1+t^2}, \; \sin(\theta)=\frac{2t}{1+t^2}$. In the implicit form it is given by  a quintic polynomial  equation.
\label{remark:Q077}
\end{remark}

Note: the above is consistent with Gibert's Q077 quintic (in barycentric coordinates) for the same locus \cite{gibert2021-q077}.
\subsection{Inparabolas}

Consider a triangle $\T:~P_i=(\cos\alpha_i,\sin\alpha_i), i=1,2,3$, inscribed in the unit circle $x^2+y^2=1$ and 
the point $F=(\cos\theta,\sin\theta)$.

\begin{remark}
The directrix of the inparabola to $\T$ with focus at $F$ is given by:
\[m x+ n y+l=0\]
where:
\begin{align*}
   m&=\cos( \theta)- \cos(\alpha_1)  -  \cos( \alpha_2) -  \cos( \alpha_3) +  \cos( \alpha_2 +  \alpha_3- \theta ) \\
   &+  \cos(\alpha_1  + \alpha_3- \theta) +  \cos(  \alpha_1 +  \alpha_2- \theta)-  \cos(  \alpha_1+ \alpha_2 +  \alpha_3 - 2  \theta )\\
  n&=   \sin( \theta) -    \sin( \alpha_1) -  \sin( \alpha_2) - \sin( \alpha_3) + \sin(  \alpha_2 +  \alpha_3- \theta )\\
   &+ \sin(  \alpha_1 +  \alpha_3- \theta) +  \sin( \alpha_1 +  \alpha_2-\theta) -  \sin(\alpha_1+ \alpha_2 +  \alpha_3 - 2  \theta)\\
  l&=3(1 -    \cos(  \alpha_1- \theta)-    \cos(  \alpha_2- \theta ) -    \cos( \alpha_3- \theta) )\\
   &+ 2  \cos( \alpha_1 -  \alpha_2)+ 2  \cos( \alpha_2 -  \alpha_3)  + 2  \cos( \alpha_1 -  \alpha_3) \\
   &-  \cos( \alpha_1 + \alpha_2 +  \alpha_3-  \theta ) -  \cos( \alpha_1-  \alpha_2 -  \alpha_3 +   \theta ) -  \cos( \alpha_1 +  \alpha_2 -  \alpha_3-  \theta ) \\
   &+  \cos( \alpha_1 +  \alpha_2-2  \theta )  
    +  \cos( \alpha_2+   \alpha_3- 2  \theta) +  \cos(  \alpha_1 +  \alpha_3-2  \theta)
\end{align*}
\label{prop:directrix}
\end{remark}
\begin{proof} The inparabola has focus on the circumrcircle and is tangent to the Simson line at the vertex \cite[Inparabola]{mw}.
By reflecting the focus $F$ about the Simson line, obtain that the directrix, known to be parallel to the Simson line, passes through point $F_1=(p/2,q/2)$ where:
\begin{align*}
   p&=  \cos(\theta)+\cos(\alpha_1)  + \cos(\alpha_2) + \cos(\alpha_3)  - \cos(\alpha_2 + \alpha_3-\theta )  - \cos(\alpha_1 + \alpha_3-\theta ) \\
   &- \cos( \alpha_1 + \alpha_2-\theta) 
    + \cos(\alpha_1+\alpha_2 + \alpha_3 - 2\theta   )\\
   q&=   \sin(\theta) +\sin(\alpha_1) + \sin(\alpha_2)+ \sin(\alpha_3) -\sin( \alpha_2 + \alpha_3-\theta ) - \sin(\alpha_1 +\alpha_3 -\theta)  \\
   &- \sin( \alpha_1 + \alpha_2-\theta) + \sin(  \alpha_1+\alpha_2 + \alpha_3 - 2 \theta )
\end{align*}
Therefore, the directrix is defined by
the equation
$\langle (x,y)-F_1, F-F_1\rangle=0$. Manipulation with a CAS yields the claim.
\end{proof}

\begin{remark}
Given a triangle $\T$, the envelope of the directrix of inparabolas with foci are points on the circumcircle is the orthocenter $X_4$ of $\T$, is given by: 
\[X_4: (\cos(\alpha_1)+\cos(\alpha_2)+\cos(\alpha_3),\sin(\alpha_1)+\sin(\alpha_2)+\sin(\alpha_3))\]
\end{remark}

\begin{proof}
The envelope of a family of lines $a(\theta)x+b(\theta)y+c(\theta)=0$ is given
by
\[ E(\theta)=\left(\frac{bc'-cb'}{ab'-a'b}, \frac{a'c-c'a}{ab'-ba}\right)\]
The result follows using CAS in the family of directrix lines given in \cref{prop:directrix}.
\end{proof}




\begin{remark}
The parametric equation of the parabola with focus $F=(x_f,y_f)$ and directrix $mx+ny+l=0$
is given by $P(t)=(x(t),y(t))$, where:
\begin{align*}
    x(t)&=\frac{(m^2 + n^2) m t^2}{ 2( m x_f +   n y_f +   l)} - n t + \frac{ m^2 x_f - ( n y_f + l) m + 2 n^2 x_f }{ 2( m^2 +   n^2)}\\
    y(t)&=\frac{(m^2 + n^2) n t^2}{2(m x_f + n y_f + l)} + m t + \frac{ n^2 y_f - ( m x_f + l) n + 2 m^2 y_f }{2( m^2 +  n^2)}
\end{align*}
The point $P(0)$ is the vertex of the parabola.
\end{remark}

\bibliographystyle{maa}
\bibliography{999_refs,999_refs_rgk,999_refs_rgk_media}

\end{document}